\documentclass[bj]{imsart}

\RequirePackage[OT1]{fontenc}
\RequirePackage[colorlinks,citecolor=blue,urlcolor=blue]{hyperref}

\usepackage{mathrsfs,amsmath,amssymb,amsthm,amsfonts,amsbsy,latexsym,dsfont,tikz}

\arxiv{arXiv:1810.11180}

\startlocaldefs
\numberwithin{equation}{section}
\theoremstyle{plain}
\newtheorem{thm}{Theorem}[section]
\theoremstyle{remark}
\newtheorem{rem}{Remark}[section]
\endlocaldefs

\newtheorem{lem}[thm]{Lemma}
\newtheorem{cor}[thm]{Corollary}
\newtheorem{prop}[thm]{Proposition}
\newtheorem{defn}[thm]{Definition}

\newtheorem{ass}[thm]{Assumption}

\newtheorem{ex}[thm]{Example}
   
\usepackage{enumerate}
\newenvironment{enumeratei}{\begin{enumerate}[\upshape (i)]}{\end{enumerate}}

\newcommand{\Var}[0]{\text{Var}}

\newcommand{\diag}[0]{\text{diag}}

\newcommand{\R}[0]{\mathbb{R}}

\newcommand{\Prob}[0]{\mathds{P}}

\newcommand{\SNR}{\mathsf{SNR}}

\newcommand{\subg}[0]{\text{sub-gaussian}}
\newcommand{\HS}[0]{\text{HS}}
\newcommand{\op}[0]{\text{op}}
\newcommand{\Image}[0]{\text{Im}}

\renewcommand{\leq}{\leqslant} 
\renewcommand{\geq}{\geqslant}

\def\qed{ \hfill $\blacksquare$}  

\newcommand{\cB}{\mathcal{B}}

\newcommand{\cG}{\mathcal{G}}

\newcommand{\cN}{\mathcal{N}}

\newcommand{\cX}{\mathcal{X}}
  
\newcommand{\vone}{\mathbf{1}}

\newcommand{\bH}{\mathbb{H}}
\newcommand{\bL}{\mathbb{L}}

\newcommand{\bR}{\mathbb{R}}
\newcommand{\bS}{\mathbb{S}}
\newcommand{\bX}{\mathbb{X}}
        
\newcommand{\dP}{\mathds{P}}

\newcommand{\sB}{\mathscr{B}}
\newcommand{\sC}{\mathscr{C}}

\DeclareMathOperator{\E}{\mathds{E}}

\DeclareMathOperator{\argmax}{argmax}

\DeclareMathOperator{\tr}{tr} 
  
\begin{document}

\begin{frontmatter}
\title{Hanson-Wright inequality in Hilbert spaces with application to $K$-means clustering for non-Euclidean data}
\runtitle{Hanson-Wright inequality in Hilbert spaces}

\begin{aug}
\author{\fnms{Xiaohui} \snm{Chen}\thanksref{a,b,e1,e3}\ead[label=e1,mark]{xhchen@illinois.edu}\ead[label=e3,mark]{xiaohui@mit.edu}}
\and
\author{\fnms{Yun} \snm{Yang}\thanksref{a,e2}\ead[label=e2,mark]{yy84@illinois.edu}}

\address[a]{Department of Statistics\newline
University of Illinois at Urbana-Champaign\newline
725 S. Wright Street, Champaign, IL 61820, USA.
\printead{e1,e2}}

\address[b]{Institute for Data, Systems, and Society\newline
Massachusetts Institute of Technology\newline
77 Massachusetts Avenue, Cambridge, MA, 02139-4307, USA.\newline
\printead{e3}}

\runauthor{X. Chen and Y. Yang}

\affiliation{Some University and Another University}

\end{aug}

\begin{abstract}
We derive a dimension-free Hanson-Wright inequality for quadratic forms of independent sub-gaussian random variables in a separable Hilbert space. Our inequality is an infinite-dimensional generalization of the classical Hanson-Wright inequality for finite-dimensional Euclidean random vectors. We illustrate an application to the generalized $K$-means clustering problem for non-Euclidean data. Specifically, we establish the exponential rate of convergence for a semidefinite relaxation of the generalized $K$-means, which together with a simple rounding algorithm imply the exact recovery of the true clustering structure.
\end{abstract}


\end{frontmatter}

\section{Introduction}
\label{sec:introduction}

The Hanson-Wright inequality is a fundamental tool for studying the concentration phenomenon for quadratic forms in sub-gaussian random variables \cite{HansonWright1971_AMS,Wright1973_AoP}. Recently, it has triggered a wide range of statistical applications such as semidefinite programming (SDP) relaxations for $K$-means clustering \cite{Royer2017_NIPS,GiraudVerzelen2018} and Gaussian approximation bounds for high-dimensional $U$-statistics (of order two) \cite{chen2018a}. Classical form of the Hanson-Wright inequality bounds the tail probability for the quadratic form of a finite-dimensional random vector in a Euclidean space. Below is a version that is frequently cited in literature (cf. Theorem 1.1 in \cite{RudelsonVershynin2013_ECP}). 

\begin{thm}[Hanson-Wright inequality for quadratic forms of independent sub-gaussian random variables in $\bR$]
\label{thm:Hanson-Wright_ineq_classical}
Let $X = (X_{1},\dots,X_{n}) \in \bR^{n}$ be a random vector with independent components $X_{i}$ such that $\E[X_{i}]=0$ and $\|X_{i}\|_{\psi_{2}} := \sup_{q \geq 1} q^{-1/2} (\E|X_{i}|^{q})^{1/q} \leq L$. Let $A$ be an $n \times n$ matrix. Then there exists a universal constant $C > 0$ such that for every $t > 0$, 
\begin{equation}
\label{eqn:Hanson-Wright_ineq_classical}
\Prob ( |X^{T} A X - \E[X^{T} A X]| \geq t ) \leq 2 \exp \left[ -C \min\left( {t^{2} \over L^{4} \|A\|_{\HS}^{2}}, {t \over L^{2} \|A\|_{\op}} \right) \right],
\end{equation}
where $\|A\|_{\HS} = (\sum_{i,j=1}^{n} a_{ij}^{2})^{1/2}$ is the Hilbert-Schmidt (i.e., Frobenius) norm of $A$ and $\|A\|_{\op} = \max_{\{x \in \bR^{n}: \|x\|_{2}=1\}} \|A x\|_{2}$ is the $\ell_{2} \to \ell_{2}$ operator (i.e., spectral) norm of $A$. 
\end{thm}

There are some variants of the finite-dimensional Hanson-Wright inequality. Sharp upper and lower tail inequalities for quadratic forms of independent Gaussian random variables are derived in \cite{LaurentMassart2000_AoS}. \cite{RauhutFoucart2013} and \cite{BartheMilman2013_CMP} derive the Hanson-Wright inequality for zero-diagonal matrix $A$ with independent Bernoulli and centered sub-gaussian random variables, respectively. \cite{HsuKakadeZhang2012_ECP} establishes an upper tail inequality for positive semidefinite quadratic forms in a sub-gaussian random vector with dependent components. \cite{VuWang2015_RSA} proves a dimension-dependent concentration inequality for a centered random vector under the convex concentration property. \cite{Adamczak2015_ECP} further improves the inequality of \cite{VuWang2015_RSA} by removing the dimension dependence in $\bR^{n}$. 

In this paper, we first derive an infinite-dimensional analog of the Hanson-Wright inequality (\ref{thm:Hanson-Wright_ineq_classical}) for sub-gaussian random variables taking values in a Hilbert space, which can be seen as a unified generalization of the aforementioned papers in finite dimensions. Motivation of deriving the dimension-free Hanson-Wright inequality stems from the generalized $K$-means clustering for non-Euclidean data with non-linear features, which covers the functional data clustering and kernel clustering as special examples. It is well-known that the (classical) Euclidean distance based $K$-means clustering is computationally $\mathsf{NP}$-hard in the worst case. Various SDP relaxations in literature (cf. \cite{PengWei2007_SIAMJOPTIM,LiLiLingStohmerWei2017,FeiChen2018,Royer2017_NIPS,GiraudVerzelen2018}) aim to provide exact and partial recovery of the true clustering structure. However, it remains a challenging task to provide strong statistical guarantees for computationally tractable (i.e., polynomial-time) algorithms to cluster non-Euclidean data taking values in a general Hilbert space with non-linear features. As we shall see in Section \ref{sec:kernel_Kmeans}, the Hilbert space version of the Hanson-Wright inequality offers a powerful tool to establish the exponential rate of convergence for an SDP relaxation of the generalized $K$-means. This partial recovery bound implies the exact recovery of the generalized $K$-means clustering via a simple rounding algorithm. In contrast to the heuristic greedy algorithms often employed in the kernel clustering setting (cf.~\cite{SongSmolaGrettonBorgwardt2007_ICML}), our result provides a principled SDP relaxed kernel clustering algorithm with exact recovery guarantees.


\section{Hanson-Wright inequality in Hilbert spaces}
\label{sec:hanson-wright_hilbert_spaces}

To state the Hanson-Wright inequality in a general Hilbert space, we first need to properly specify the sub-gaussian random variables therein. 

\subsection{Sub-gaussian random variables in Hilbert spaces}

Let $\bH$ be a real separable Hilbert space and $\cB(\bH)$ be the class of bounded linear operators $\Sigma: \bH \to \bH$. If the operator $\Sigma \in \cB(\bH)$ is positive definite (i.e., it is self-adjoint $\Sigma^{*} = \Sigma$ and  $\langle \Sigma z, z \rangle \geq 0$ for all $z \in \bH$), then there is a unique positive definite (and thus self-adjoint) square root operator $\Sigma^{1/2} \in \cB(\bH)$ satisfying $\Sigma^{1/2} \Sigma^{1/2} = \Sigma$ (cf. Theorem 3.4.3 in \cite{HsingEubank2015_Wiley}). 

\begin{defn}[Trace class of linear operators on a separable Hilbert space]
Let $\Sigma \in \cB(\bH)$. Then $\Sigma$ is {\it trace class} if 
\[
\|\Sigma\|_{\tr} := \sum_{j=1}^{\infty} \langle (\Sigma^{*}\Sigma)^{1/2} e_{j}, e_{j} \rangle < \infty,
\]
where $(e_{j})_{j=1}^{\infty}$ is a {\it complete  orthonormal system} (CONS) of $\bH$. In this case, $\|\Sigma\|_{\tr}$ is the {\it trace norm} of $\Sigma$. 
\end{defn}

Note that the trace norm does not depend on the choice of the CONS. A self-adjoint and positive definite trace class linear operator $\Sigma$ is compact and it plays a similar role as a covariance matrix, where the trace norm is simply the trace of the covariance matrix. In particular, if $\Sigma$ is positive definite trace class, then $\|\Sigma\|_{\tr} = \sum_{j=1}^{\infty} \langle \Sigma e_{j}, e_{j} \rangle = \sum_{j=1}^{\infty} \| \Sigma^{1/2} e_{j} \|^{2}$. Let $(\Omega, \sB, \Prob)$ be a probability space. 

\begin{defn}[Hilbert space valued sub-gaussian random variable]
\label{defin:subgaussian_Hilbert_space}
Let $Z$ be a random variable in $\bH$ and $\Gamma : \bH \to \bH$ be a positive definite trace class linear operator. Then $Z$ is {\it sub-gaussian with respect to $\Gamma$} (denote as $Z \sim \subg(\Gamma)$) if there exists an $\alpha \geq 0$ such that for all $z \in \bH$, 
\begin{equation}
\label{eqn:subgaussian_Hilbert_space}
\E\left[ e^{\langle z, Z - \E[Z] \rangle} \right] \leq e^{\alpha^{2} \langle \Gamma z, z \rangle / 2},
\end{equation} 
where the expectation $\E[Z] = \int_{\Omega} Z d\dP$ is defined as a Bochner integral (cf. Chapter 2.6 in \cite{HsingEubank2015_Wiley}). Moreover, if $Z \sim \subg(\Gamma)$ with mean $\mu=\E[Z]$, then the {\it $\psi_{2}$ (or sub-gaussian) norm} of $Z$ with respect to $\Gamma$ is defined as 
\[
\|Z\|_{\psi_{2,\Gamma}} = \inf \left\{ \alpha \geq 0 : \E\left[ e^{\langle z, Z-\mu \rangle} \right] \leq e^{\alpha^{2} \langle \Gamma z, z \rangle / 2} \quad \forall z \in \bH \right\}.
\]
\end{defn}
Note that Definition \ref{defin:subgaussian_Hilbert_space} corresponds to the $R$-sub-gaussianity in \cite{Antonini1997_RSMUP}, and it is an infinite-dimensional analog of the sub-gaussian random vectors in $\bR^{p}$ (see for example \cite{Vershynin2018_Cambridge} and \cite{HsuKakadeZhang2012_ECP}). Unsurprisingly, the Gaussian random variables in $\bH$ is a special case of sub-gaussian random variables in $\bH$. 

\begin{defn}[Hilbert space valued Gaussian random variable]
\label{defin:gaussian_Hilbert_space}
A random variable $Z$ in $\bH$ is {\it Gaussian with respect to $\Gamma$} and with mean $\mu=\E[Z]$ (denote as $Z \sim N(\mu, \Gamma)$) if for all $z \in \bH$, 
\begin{equation}
\label{eqn:gaussian_Hilbert_space}
\E\left[ e^{\langle z, Z-\mu \rangle} \right] = e^{\langle \Gamma z, z \rangle / 2}.
\end{equation}
\end{defn}

\begin{lem}
\label{lem:hilbert_space_gaussian_rv}
If $Z \sim N(\mu, \Gamma)$, then $\|Z\|_{\psi_{2}, \Gamma} = 1$ and $\Sigma = \Gamma$, where $\Sigma := \E[(Z-\mu) \otimes (Z-\mu)]$ is the covariance operator of $Z$. More generally, if $Z \sim \subg(\Gamma)$ with mean $\mu=\E[Z]$, then $\Sigma \preceq 4 \|Z\|_{\psi_{2}, \Gamma}^{2} \Gamma$, i.e., $(4 \|Z\|_{\psi_{2}, \Gamma}^{2} \Gamma - \Sigma)$ is positive semidefinite. 
\end{lem}

For $a, b \in \bH$, the tensor product $a \otimes b : \bH \to \bH$ is a linear operator defined as $(a \otimes b) z = \langle b, z \rangle a$ for all $z \in \bH$. Lemma~\ref{lem:hilbert_space_gaussian_rv} is proved in Appendix~\ref{sec:auxiliary_lemmas}. \\

\noindent {\bf Notation.} We shall use $c, c_{0}, c_{1}, C, C_{0}, C_{1}, \dots$ to denote positive and finite universal constants, whose values may vary from place to place. For $a, b \in \bR$, denote $a \vee b = \max(a, b)$ and $a \wedge b = \min(a, b)$. For $\Sigma \in \cB(\bH)$, the operator norm $\|\Sigma\|_{\op}$ of $\Sigma$ is defined as the square root of the largest eigenvalue of $\Sigma^{*} \Sigma$. If $\sum_{j=1}^{\infty} \|\Sigma e_{j}\|^{2} < \infty$, then $\Sigma$ is a Hilbert-Schmidt (HS) operator and $\|\Sigma\|_{\HS} = (\sum_{j=1}^{\infty} \|\Sigma e_{j}\|^{2})^{1/2}$. For a matrix $Z \in \bR^{m \times n}$, $|Z|_{1} = \sum_{i=1}^{m} \sum_{j=1}^{n} |Z_{ij}|$. 

\subsection{Hanson-Wright inequality in Hilbert spaces}
\label{subsec:hanson-wright_hilbert_space}

Throughout Section \ref{subsec:hanson-wright_hilbert_space}, we assume that $\bH$ is a real separable Hilbert space and $\Gamma \in \cB(\bH)$ is a positive definite trace class operator on $\bH$. First, we present a Hanson-Wright inequality with zero diagonal in Proposition \ref{prop:Hanson-Wright_ineq_Hilbert_space_diagonal_free}. 

\begin{prop}[Hanson-Wright inequality for quadratic forms of sub-gaussian random variables in Hilbert spaces: zero diagonal]
\label{prop:Hanson-Wright_ineq_Hilbert_space_diagonal_free}
Let $X_{i}, i=1,\dots,n$, be a sequence of independent centered $\subg(\Gamma)$ random variables in $\bH$ and $L_{i} = \|X_{i}\|_{\psi_{2},\Gamma}$. Let $A = (a_{ij})_{i,j=1}^{n}$ be an $n \times n$ matrix and $S = \sum_{1 \leq i \neq j \leq n} a_{ij} \langle X_{i}, X_{j} \rangle$. Then there exists a universal constant $C > 0$ such that for any $t > 0$, 
\begin{equation}
\label{eqn:Hanson-Wright_ineq_Hilbert_space_diagonal_free}
\Prob \left( S \geq t \right ) \leq \exp\left[ -C \min\left( {t^{2} \over L^{4} \|\Gamma\|_{\HS}^{2} \|A\|_{\HS}^{2}}, {t \over L^{2} \|\Gamma\|_{\op} \|A\|_{\op}} \right)\right],
\end{equation}
where $L = \max_{1 \leq i \leq n} L_{i}$. 
\end{prop}

\begin{rem}
Proposition \ref{prop:Hanson-Wright_ineq_Hilbert_space_diagonal_free} is a dimension-free version of the Hanson-Wright inequality with a zero diagonal weighting matrix for {\it independent} sub-gaussian random variables in $\bR$ \cite{RudelsonVershynin2013_ECP}. Specifically, Theorem \ref{thm:Hanson-Wright_ineq_classical} (i.e., Theorem 1.1 in \cite{RudelsonVershynin2013_ECP}) is a special case of Proposition \ref{prop:Hanson-Wright_ineq_Hilbert_space_diagonal_free} with $\bH = \bR$ and $\langle X_{i}, X_{j} \rangle = X_{i} X_{j}$. In this case, we may take $\Gamma = 1$ and thus $\|\Gamma\|_{\op} = \|\Gamma\|_{\HS} = 1$. Different from Theorem \ref{thm:Hanson-Wright_ineq_classical}, Proposition \ref{prop:Hanson-Wright_ineq_Hilbert_space_diagonal_free} is also able to capture the (component-wise) dependency encoded in $\Gamma$ for general Hilbert spaces, thus covering certain quadratic forms in a finite-dimensional sub-gaussian random vector with {\it dependent} components. We emphasize that, although our general proof strategy of decoupling the off-diagonal dependence is based on that of Theorem 1.1 in \cite{RudelsonVershynin2013_ECP}, a key step in our proof to remove the dependency in the Hilbert space valued sub-gaussian random variables is diagonalizing the operator $\Gamma$ (together with the decoupling). Such diagonalization procedure allows us to perform the calculations in an isometric $\ell^{2}$ space of $\bH$, where linear operators can be conveniently represented by (infinite-dimensional) matrices. This turns out to be the crux to obtain the trade-off between $\|\Gamma\|_{\HS}$ and $\|\Gamma\|_{\op}$ in the tail probability bound for the off-diagonal sum $S$.  \qed
\end{rem}

Our next result is an upper tail inequality (i.e., one-sided Hanson-Wright inequality) with non-negative diagonal weights in Theorem \ref{thm:Hanson-Wright_ineq_Hilbert_space_nonnegative_diagonal} below. 

\begin{thm}[Upper tail inequality for quadratic forms of sub-gaussian random variables in Hilbert spaces: non-negative diagonal]
\label{thm:Hanson-Wright_ineq_Hilbert_space_nonnegative_diagonal}
Let $X_{i}, i=1,\dots,n$, be a sequence of independent centered $\subg(\Gamma)$ random variables in $\bH$ and $L_{i} = \|X_{i}\|_{\psi_{2},\Gamma}$. Let $A = (a_{ij})_{i,j=1}^{n}$ be an $n \times n$ matrix such that $a_{ii} \geq 0$, and $Q = \sum_{i,j=1}^{n} a_{ij} \langle X_{i}, X_{j} \rangle$. Then there exists a universal constant $C > 0$ such that for any $t > 0$, 
\begin{equation}
\label{eqn:Hanson-Wright_ineq_Hilbert_space_nonnegative_diagonal}
\Prob \left( Q \geq \sum_{i=1}^{n} a_{ii} L_{i}^{2} \|\Gamma\|_{\tr} + t \right ) \leq 2 \exp\left[ -C \min\left( {t^{2} \over L^{4} \|\Gamma\|_{\HS}^{2} \|A\|_{\HS}^{2}}, {t \over L^{2} \|\Gamma\|_{\op} \|A\|_{\op}} \right)\right],
\end{equation}
where $L = \max_{1 \leq i \leq n} L_{i}$. 
\end{thm}

Both Proposition~\ref{prop:Hanson-Wright_ineq_Hilbert_space_diagonal_free} and Theorem~\ref{thm:Hanson-Wright_ineq_Hilbert_space_nonnegative_diagonal} allow $X_{i}, i = 1,\dots,n$, to have different covariance operators $\Sigma_{i}$, provided that $\Sigma_{i} \preceq 4 \|X_{i}\|_{\psi_{2}, \Gamma}^{2} \Gamma$ (cf. Lemma \ref{lem:hilbert_space_gaussian_rv}). 

\begin{rem}[Connections to the existing upper tail inequality in finite-dimensional Euclidean spaces]
First, we mention that the upper tail probability bound~\eqref{eqn:Hanson-Wright_ineq_Hilbert_space_nonnegative_diagonal} (also cf. Lemma \ref{lem:mgf_bound_subgaussian_centered}) is sharper than the one-dimensional Bernstein's inequality for the non-negatively weighted diagonal sum of squared norm of independent sub-gaussian random variables in $\bH$. Indeed, if we simply apply Bernstein's inequality (cf. Theorem 2.8.1 in \cite{Vershynin2018_Cambridge}) for the real-valued sub-exponential random variables $\|X_{i}\|^{2}$ (cf. Lemma \ref{lem:squared_norm_is_subexp}), then the diagonal sum in $Q$ has the following probability bound: for any $t > 0$, 
\begin{equation}
\label{lem:Hanson-Wright_ineq_Hilbert_space_linear_part_less_sharp}
\begin{split}
& \Prob \left( \left| \sum_{i=1}^{n} a_{ii} (\|X_{i}\|^{2} - \E \|X_{i}\|^{2}) \right| \geq t \right) \\
& \qquad \leq 2 \exp \left[ -C \min \left( {t^{2} \over L^{4} \|\Gamma\|_{\tr}^{2} \sum_{i=1}^{n} a_{ii}^{2} }, {t \over L^{2} \|\Gamma\|_{\tr} \max_{1 \leq i \leq n} |a_{ii}|} \right) \right].
\end{split}
\end{equation}
Note that the right-hand side of (\ref{lem:Hanson-Wright_ineq_Hilbert_space_linear_part_less_sharp}) is controlled by one parameter $\|\Gamma\|_{\tr}$, which is strictly less sharp than (\ref{eqn:Hanson-Wright_ineq_Hilbert_space_nonnegative_diagonal}) since $\|\Gamma\|_{\op} \leq \|\Gamma\|_{\tr}$ and $\|\Gamma\|_{\HS}^{2} \leq \|\Gamma\|_{\op} \|\Gamma\|_{\tr} \leq \|\Gamma\|_{\tr}^{2}$. For instance, if $X_{i} \in \R^{p}$, then $\Gamma$ is often the $p \times p$ covariance matrix of $X_{i}$. In the special case for $\Gamma = I_{p}$, then $\|\Gamma\|_{\op} = 1$, $\|\Gamma\|_{\HS} = p^{1/2}$, and $\|\Gamma\|_{\tr} = p$. Therefore, direct application of the diagonal sum bound (\ref{lem:Hanson-Wright_ineq_Hilbert_space_linear_part_less_sharp}) does not yield the probability bound in Proposition \ref{prop:Hanson-Wright_ineq_Hilbert_space_diagonal_free}. In particular, for the generalized $K$-means clustering problem, this implies that a much more restrictive lower bound condition on the signal-to-noise ratio is required for exact recovery of the true clustering structure for high-dimensional data (more details can be found in the discussion after Theorem~\ref{thm:rate_SDP_kernel_kmeans_general_case}). 

Second, for non-negative diagonal weights, Theorem \ref{thm:Hanson-Wright_ineq_Hilbert_space_nonnegative_diagonal} is an infinite-dimensional (and thus dimension-free) generalization of the tail inequality for quadratic forms a sub-gaussian random vector with dependent components in $\bR^{p}$ \cite{HsuKakadeZhang2012_ECP}. In particular, if $X = (X_{1},\dots,X_{p})$ is a centered sub-gaussian random vector in $\bR^{p}$ (i.e., there exists a $\sigma \geq 0$ such that $\E[e^{z^{T}X}] \leq e^{\|z\|_{2}^{2} \sigma^{2} / 2}$ for all $z \in \bR^{p}$), then Theorem 2.1 in \cite{HsuKakadeZhang2012_ECP} states that: for any positive semidefinite matrix $\Sigma$ and $t > 0$, 
\[
\Prob \left( X^{T} \Gamma X \geq \sigma^{2} (\|\Gamma\|_{\tr} + 2 \|\Gamma\|_{\HS} \sqrt{t} + 2 \|\Gamma\|_{\op} t ) \right) \le e^{-t}. 
\]
The last inequality is a special case (up to a universal constant) of (\ref{eqn:Hanson-Wright_ineq_Hilbert_space_nonnegative_diagonal}) with $n = 1$, $A = 1$, $\bH = \bR^{p}$, $\Gamma^{-1/2} X \sim \subg(\sigma^{2} I_{p})$, and $L^{2} = \sigma^{2}$. In addition, we note that the positive semidefinite condition is not needed in our Theorem \ref{thm:Hanson-Wright_ineq_Hilbert_space_nonnegative_diagonal}. Instead, only a weaker condition on the non-negativity of the diagonal entries in the weighting matrix is required. \qed
\end{rem}

There are two limitations of Theorem \ref{thm:Hanson-Wright_ineq_Hilbert_space_nonnegative_diagonal}. First, $Q$ is typically not centered at $\sum_{i=1}^{n} a_{ii} L_{i}^{2} \|\Gamma\|_{\tr}$. For the generalized $K$-means application in Section \ref{sec:kernel_Kmeans}, this means that consistency of solutions of the SDP relaxation (\ref{eqn:clustering_Kmeans_sdp}) cannot be attained unless $\sum_{i=1}^{n} a_{ii} L_{i}^{2} \|\Gamma\|_{\tr}$ tends to $\E[Q]$. Second, the non-negativity condition on the diagonal weights $a_{ii} \geq 0$ in Theorem \ref{thm:Hanson-Wright_ineq_Hilbert_space_nonnegative_diagonal} is not entirely innocuous for obtaining a concentration inequality for $Q$ (i.e., two-sided Hanson-Wright inequality). Without imposing additional assumptions, we cannot expect a lower tail bound for sub-gaussian random variables even in $\bR^{n}$ \cite{Adamczak2015_ECP}. To simultaneously fix these two issues and obtain a concentration inequality for $Q-\E[Q]$, we make the following Bernstein-type assumption on the squared norm, in addition to the assumption that $X_{1},\dots,X_{n}$ are independent $\subg(\Gamma)$ with mean zero. 
\begin{ass}[Bernstein condition on the squared norm]
\label{ass:bernstein_squared_norm}
There exists a universal constant $C > 0$ such that 
\begin{equation}
\label{eqn:bernstein_squared_norm}
\E \left| \|X_{i}\|^{2} - \E\|X_{i}\|^{2} \right|^{k} \leq C k! L_{i}^{k-2} \|\Gamma\|_{\op}^{k-2} \|\Sigma_{i}\|_{\HS}^{2}  \quad \forall k = 3,4,\dots,
\end{equation}
where $\Sigma_{i} = \E[X_{i} \otimes X_{i}]$ is the covariance operator of $X_{i}, i = 1,\dots,n$. 
\end{ass}

\begin{rem}[Comments on Assumption \ref{ass:bernstein_squared_norm}]
Since $\|\Sigma_{i}\|_{\tr} = \E\|X_{i}\|^{2}$, Assumption \ref{ass:bernstein_squared_norm} is a mild condition on the sub-exponential tail behavior of $\|X_{i}\|^{2} - \|\Sigma_{i}\|_{\tr}$. For $\bH = \bR$, (\ref{eqn:bernstein_squared_norm}) is an automatic consequence of the sub-gaussianality (\ref{eqn:subgaussian_Hilbert_space}). For $\bH = \bR^{p}$, if $X = \Sigma^{1/2} Z$, where $Z = (Z_{1}, \dots, Z_{p})^{T}$ has independent components $Z_{j}$ with bounded sub-gaussian norms, then  
\[
\E[\|X\|^{2} - \E\|X\|^{2}]^{2} = \E[Z^{T} \Sigma Z - \tr(\Sigma)]^{2} \lesssim \|\Sigma\|_{\HS}^{2}.
\]
Such linear transformation of an independent random vector in $\bR^{p}$ with sub-gaussian components is a popular statistical model for the $K$-means clustering \cite{GiraudVerzelen2018,Royer2017_NIPS}. For the general Hilbert space $\bH$, it is easy to verify that Gaussian random variable $Z \sim N(0, \Gamma)$ in $\bH$ satisfies (\ref{eqn:bernstein_squared_norm}). 
Comparing with the ``centering" term $\sum_{i=1}^{n} a_{ii} L_{i}^{2} \|\Gamma\|_{\tr}$ in (\ref{eqn:Hanson-Wright_ineq_Hilbert_space_nonnegative_diagonal}), we shall see that the correct centering terms $\E\|X_{i}\|^{2}$ in (\ref{eqn:bernstein_squared_norm}) together with the parameters $(L_{i} \|\Gamma\|_{\op}, \|\Sigma_{i}\|_{\HS})$ are crucial to yield a concentration inequality for $Q-\E[Q]$. By Lemma \ref{lem:hilbert_space_gaussian_rv}, we know that $4 L_{i}^{2} \|\Gamma\|_{\tr} \geq \|\Sigma_{i}\|_{\tr}$ for any $X_{i} \sim \subg(\Gamma)$. In fact, even in $\bR$, it is easy to construct a random variable $X \sim \subg(\gamma^{2})$ such that $\gamma^{2} \gg \sigma^{2}$ where $\sigma^{2} = \Var(X)$ (cf. Example 4.1 and 4.2 in \cite{chen2018a}). In particular, here we give a counterexample in $\bR$ (so that $L_{i} = 1$). Let $Y_{n}$ follow a mixture of Gaussian distributions $F_{n} = (1-\epsilon_{n}) N(0,1) + \epsilon_{n} N(0, a_{n}^{2})$, where $a_{n} > 1$ and $\epsilon_{n} = a_{n}^{-4}$.  Then we have $\sigma_{n}^{2} := \Var(Y_{n}) = 1 - a_{n}^{-4} + a_{n}^{-2}$ and $Y_{n} \sim \subg(\gamma_{n}^{2})$, where $\gamma_{n}^{2} = C a_{n}^{2}$ for some sufficiently large constant $C > 0$. Thus if $a_{n} \to \infty$ as $n \to \infty$, then $\sigma_{n}^{2} \asymp 1$ and 
\[
\E|Y_{n}^{2} - \E Y_{n}^{2} |^{k} \lesssim a_{n}^{2k-4} \E|Z|^{2k} = a_{n}^{2k-4} (2k-1)!! \leq 4 k! (2a_{n}^{2})^{k-2} \lesssim k! (\gamma_{n}^{2})^{k-2} (\sigma_{n}^{2})^{2}, 
\]
where $Z \sim N(0,1)$. Hence $(Y_{n})_{n=1,2,\dots}$ is a sub-gaussian random variable satisfying Assumption \ref{ass:bernstein_squared_norm} and $\sigma_{n}^{2} \ll \gamma_{n}^{2}$, provided that $a_{n} \to \infty$ as $n \to \infty$. 
\qed
\end{rem}

Now we are ready to state the Hanson-Wright inequality for the general case. 

\begin{thm}[Hanson-Wright inequality for quadratic forms of sub-gaussian random variables in Hilbert spaces: general version]
\label{thm:HW_ineq_Hilbert_space}
Let $X_{i}, i=1,\dots,n$, be a sequence of independent centered $\subg(\Gamma)$ random variables in $\bH$ and $L_{i} = \|X_{i}\|_{\psi_{2},\Gamma}$. Let $A = (a_{ij})_{i,j=1}^{n}$ be an $n \times n$ matrix and $Q = \sum_{i,j=1}^{n} a_{ij} \langle X_{i}, X_{j} \rangle$. If in addition Assumption \ref{ass:bernstein_squared_norm} holds, then there exists a universal constant $C > 0$ such that for any $t > 0$, 
\begin{equation}
\label{eqn:HW_ineq_Hilbert_space}
\Prob \left( \left| Q - \E[Q] \right| \geq t \right) \leq 2 \exp\left[ -C \min\left( {t^{2} \over L^{4} \|\Gamma\|_{\HS}^{2} \|A\|_{\HS}^{2}}, {t \over L^{2} \|\Gamma\|_{\op} \|A\|_{\op}} \right)\right], 
\end{equation}
where $L = \max_{1 \leq i \leq n} L_{i}$. 
\end{thm}

\cite{VuWang2015_RSA} and \cite{Adamczak2015_ECP} derive Hanson-Wright inequalities under the convex concentration property of a finite-dimensional random vector, which is difficult to verify in general. In contrast, our Theorem \ref{thm:HW_ineq_Hilbert_space} holds under more transparent conditions (i.e., the sub-gaussian and Bernstein-type assumptions). Note that Theorem \ref{thm:HW_ineq_Hilbert_space} can be seen as a unified generalization of the finite-dimensional Hanson-Wright inequality to Hilbert spaces for both {\it independent} sub-gaussian random variables in $\bR$ \cite{RudelsonVershynin2013_ECP} and a sub-gaussian random vector with {\it dependent} components in $\bR^{p}$ \cite{HsuKakadeZhang2012_ECP}.

\section{$K$-means clustering in Hilbert spaces and its semidefinite relaxation}
\label{sec:kernel_Kmeans}

In this section, we apply the Hanson-Wright inequality in Section \ref{subsec:hanson-wright_hilbert_space} (i.e., Theorem \ref{thm:HW_ineq_Hilbert_space}) to the clustering problem of $n$ data points into $K$ clusters such that $K \leq n$. Let $X_{1},\dots,X_{n}$ be a sequence of independent random variables taking values in a measurable space $(\bX, \cX)$ on $(\Omega, \sB, \Prob)$. Suppose that there exists a clustering structure $G_{1}^{*},\dots,G_{K}^{*}$  (i.e., a partition on $[n] := \{1,\dots,n\}$ satisfying $\cup_{k=1}^{K} G_{k}^{*} = \{1,\dots,n\}$ and $G_{k}^{*} \cap G_{m}^{*} = \emptyset$ if $1 \leq k \neq m \leq K$) on the $n$ data points with $X_{i} \sim P_{k}$ for $i \in G_{k}^{*}$, where $P_{1},\dots,P_{K}$ are distinct distributions on $(\bX, \cX)$. We emphasize that $\bX$ does not need to be a Euclidean space. Our goal is to develop a statistically correct and computationally tractable algorithm for recovering the true clustering structure based on the similarity of the observations $X_{1},\dots,X_{n}$. 

\subsection{$K$-means in Hilbert spaces: 0-1 integer program formulation} 

Perhaps one of the most widely used clustering methods is the Euclidean distance-based $K$-means clustering, due to the existence of computationally efficient heuristic algorithms (such as Lloyd's algorithm \cite{Lloyd1982_TIT}). This is a particularly attractive feature for large datasets. Given a sequence of observations $X_{1},\dots,X_{n} \in \R^{p}$ (i.e., $\bX = \bR^{p}$), the (classical) $K$-means clustering method minimizes the total intra-cluster squared Euclidean distances 
\[
\min_{G_{1},\dots,G_{K}} \sum_{k=1}^{K} {1 \over |G_{k}|} \sum_{i,j \in G_{k}} \|X_{i}-X_{j}\|^{2}
\]
over all possible partitions on $[n]$, where $|G_{k}|$ is the cardinality of $G_{k}$. Dropping the sum of squared norms $\sum_{i=1}^{n} \|X_{i}\|^{2}$, we see that the $K$-means clustering is equivalent to the maximization of the total intra-cluster correlations 
\[
\max_{G_{1},\dots,G_{K}} \sum_{k=1}^{K} {1 \over |G_{k}|} \sum_{i,j \in G_{k}} X_{i}^{T} X_{j}.
\]
Here, $X_{i}^{T} X_{j}$ can be viewed as a similarity measure specified by the Euclidean space inner product $a_{ij} = \langle X_{i}, X_{j}\rangle_{\bR^p}$. In general, if space $\bX$ is a Hilbert space $\bH$, then it is natural to generalize this procedure by replacing $\langle\cdot,\cdot\rangle_{\bR^p}$ with the inner product $\langle\cdot,\,\cdot\rangle_{\bH}$ associated with $\bH$, yielding $a_{ij} = \langle X_{i},X_{j}\rangle_{\bH}$.
Henceforth, we will refer to such a $K$-means that uses the inner product in a Hilbert space as a generalized $K$-means. 

\begin{ex}[Functional data clustering]
\label{ex:functional_data}
In many applications, data to be clustered are recorded as curves, surfaces or other things varying over a continuum, such as a time interval and a space span. The random variable underlying data is naturally modelled as a stochastic process $X=\{X(t):\,t\in\mathcal T\}$ in Hilbert space $(\bH,\langle\cdot,\cdot\rangle_{\bH})$, where the sequence of observations $X_1,\ldots,X_n\in\bH$ is an i.i.d.~sample of random variables drawn from the same distribution as $X$. In clustering problems, the law of $X$ is often assumed to be a mixture distribution over $\bH$, with each mixture component as a cluster. When $\mathcal T=[0,1]$ is the unit interval, we can choose $\bH$ as the $L^2$ function space $\bL^2[0,1] =\{f:[0,1]\to \bR:\,\|f\|_{\bL^2}^2 =\int_0^1|f(t)|^2\,dt<\infty\}$ with $\bL^2$-inner product $\langle f,g\rangle_{\bL^2}=\int_0^1 f(t)g(t)\,dt$ for $f,g\in\bL^2[0,1]$. Suppose we have prior information that the observations $\{X_i\}$ are smooth functions, then we can choose a stronger norm to capture the similarity in the (higher-order) derivatives. For example, in \cite{ieva2013multivariate,tarpey2003clustering} and \cite{ferraty2006nonparametric}, $\bH$ are recommended as the Sobolev space with some order $k\in\{1,2\}$ as $\bS^k[0,1]=\{f:[0,1]\to \bR:\,\|f^{(k)}\|_{\bL^2}^2 =\int_0^1|f^{(k)}(t)|^2\,dt<\infty\}$ equipped with inner product $\langle f,g\rangle_{\bS^k} =\sum_{j=0}^k \langle f^{(j)},g^{(j)}\rangle_{\bL^2}$, where $f^{(k)}$ denotes the $k$th derivative of a function $f\in \bS^k[0,1]$. As we will see in Section~\ref{sec:application_FDA}, a higher smoothness order $k$ in the generalized $K$-means generally leads to larger separations among cluster centers (between cluster variation) without significantly increasing fluctuations within clusters (within cluster variation), thereby increasing the clustering signal-to-noise ratio (see Theorem~\ref{thm:rate_SDP_kernel_kmeans_general_case} for a precise definition). 
\qed
\end{ex}

\begin{ex}[Kernel clustering]
\label{ex:polynomial_kernels}
In pattern recognition and natural language processing, it is often crucial to capture the non-linear similarity for non-Euclidean data (such as images and words). A widely used approach is the kernel method \cite{ScholkopfSmola2001_LearningKernels}, where the similarity $a_{ij}$ between $X_i$ and $X_j$ is characterized by a nonlinear positive semi-definite kernel function $\rho:\bX\times\bX \to \bR$ through $a_{ij}=\rho(X_i,X_j)$. Commonly used kernel functions include polynomial kernels
$\rho(x, y) = (\langle x, y \rangle + c)^{r}$ for some positive integer order $r$ and radial basis function (RBF) kernel $\rho(x,y)=\exp\{-\|x-y\|^2/(2h^2)\}$ for some bandwidth parameter $h>0$, where $x,y\in\bR^p$ are the Euclidean embeddings of the original observations (image pixel level vectorizations or word embeddings). According to the celebrated Mercer's theorem, kernel clustering can also be viewed as $K$-means in a high-dimensional feature space: there always exists a Hilbert space (feature space) $\bH$ equipped with inner product $\langle\cdot,\cdot\rangle_{\bH}$ and a feature map $\phi:\bX\rightarrow \bH$, such that 
\begin{align*}
\rho(x,y) = \langle \phi(x),\phi(y)\rangle_{\bH},\quad\forall x,y\in\bX.
\end{align*}
More details about a construction of the feature map can be found in Section~\ref{sec:feature_map}.
From this identity, kernel $K$-means that uses a nonlinear similarity measure $a_{ij}=\rho(X_i,X_j)$ can be cast into the framework of $K$-means in Hilbert spaces by identifying $X_i$ as $\phi(X_i)$. On the other hand, explicit representations for the feature map $\phi$ and the Hilbert space $\bH$ are not necessary in order to implement the kernel $K$-means, which is one of the main practical attractiveness of the method. By choosing a proper kernel $\rho$, we may capture the non-linear similarity in non-Euclidean spaces through implicitly mapping the original data space $\bX$ into a ``high-dimensional" feature space, in which linear boundaries can be drawn to separate the data points.
For example, the polynomial kernel maps into the space spanned by the products of all monomials up to degree $r$. In particular, clusters with centers (expectations under $P_j$'s) that are overlapped in the original Euclidean space may have separated centers (expectations under $\phi_{\#}(P_j)$'s, where $\phi_\#(\mu)$ denotes the pushforward of measure $\mu$ defined through $(\phi_{\#}(\mu))(B)=\mu(\phi^{-1}(B))$ for every measurable subset $B\subset\bH$) in the feature space. \qed
\end{ex}

For a general inner product $\langle\cdot,\cdot\rangle_{\bH}$, quadratic sample complexity is needed for the generalized $K$-means to compute the similarity matrix $A$ \cite{FilipponeCamastraMasulliRovetta2008_PR}. Observe that, for every partition $G_{1},\dots,G_{K}$, there is a one-to-one $n \times K$ {\it assignment matrix} $H = (h_{ik}) \in \{0,1\}^{n \times K}$ such that $h_{ij} = 1$ if $i \in G_{k}$ and $h_{ij} = 0$ if $i \notin G_{k}$. Thus the $K$-means clustering problem can be written as a 0-1 integer program:
\begin{equation}
\label{eqn:kernel_Kmeans_integer_program}
\max \left\{ \langle A, H B H^{T} \rangle : H \in \{0,1\}^{n \times K}, H \vone_{K} = \vone_{n} \right\},
\end{equation}
where $\vone_{n}$ denotes the $n \times 1$ vector of all ones, $a_{ij} = \langle X_{i}, X_{j}\rangle_{\bH}$, and $B = \diag(|G_{1}|^{-1},\dots,|G_{K}|^{-1})$. 

The generalized $K$-means clustering problem (\ref{eqn:kernel_Kmeans_integer_program}) is typically computationally intractable, namely polynomial-time algorithms with exact solutions only exist in certain cases \cite{SongSmolaGrettonBorgwardt2007_ICML}. For instances, the (classical) $K$-means clustering is a worst-case $\mathsf{NP}$-hard integer programming problem with a non-linear objective function \cite{PengWei2007_SIAMJOPTIM}. Exact and partial recovery properties of various SDP relaxations for the $K$-means \cite{PengWei2007_SIAMJOPTIM,LiLiLingStohmerWei2017,FeiChen2018,Royer2017_NIPS,GiraudVerzelen2018} are studied in literature. However, it remains a challenging task to provide statistical guarantees for the generalized $K$-means clustering to capture the non-linear features of non-Euclidean data taking values in a general Hilbert space. 

\subsection{SDP relaxation for $K$-means in Hilbert spaces}
\label{subsec:sdp_kernel_Kmeans}

We consider the SDP relaxations for the generalized $K$-means clustering. Note that every partition $G_{1},\dots,G_{K}$ of $[n]$ can be represented by a partition function $\sigma : [n] \to [K]$ via $G_{k}=\sigma^{-1}(k), k=1,\dots,n$. If we change the variable $Z = H B H^{T}$ in the 0-1 integer program formulation (\ref{eqn:kernel_Kmeans_integer_program}) of the generalized $K$-means, then $Z$ satisfies the following properties: 
\begin{equation}
\label{eqn:constraints_clustering_generic_integer_program}
Z^{T} = Z, \quad Z \succeq 0, \quad \tr(Z) = \sum_{k=1}^{K} |G_{k}| \, b_{kk}, \quad (Z \vone_{n})_{i} = \sum_{k=1}^{K} |G_{k}| \, b_{\sigma(i)k}, \; i=1,\dots,n.
\end{equation}
For the generalized $K$-means $B = \diag(|G_1|^{-1},\dots,|G_{K}|^{-1})$, the last constraint in (\ref{eqn:constraints_clustering_generic_integer_program}) reduces to $Z \vone_{n} = \vone_{n}$, which does not depend on the partition function $\sigma$. Thus we can relax the generalized $K$-means clustering to the SDP problem: 
\begin{equation}
\label{eqn:clustering_Kmeans_sdp}
\hat{Z} = \argmax \left\{ \langle A, Z \rangle : Z \in \sC \right\} \text{ with } \sC = \{ Z^{T} = Z, Z \succeq 0, \tr(Z) = K, Z \vone_{n} = \vone_{n}, Z \geq 0 \},
\end{equation}
where $Z \succeq 0$ means that $Z$ is positive semidefinite and $Z \geq 0$ means that all entries of $Z$ are non-negative. We shall use $\hat{Z}$ to estimate the true ``membership matrix" $Z^{*}$, where 
\begin{equation}
\label{eqn:Kmeans_true_membership_matrix}
Z_{ij}^{*} = \left\{
\begin{array}{cc}
1/n_{k} & \text{if } i, j \in G_{k}^{*} \\
0 & \text{otherwise} \\
\end{array}
\right. ,
\end{equation}
where $n_k = |G^*_k|$. Note that $Z^{*} \in \sC$ is a projection matrix such that $Z^{*} Z^{*} = Z^{*}$. If $X_{1},\dots,X_{n} \in \bR^{p}$ (i.e., $\bX = \bR^{p}$) and $a_{ij} = X_{i}^{T} X_{j}$ is the Euclidean space inner product, then (\ref{eqn:clustering_Kmeans_sdp}) is the SDP proposed in \cite{PengWei2007_SIAMJOPTIM}. Observe that the SDP relaxation (\ref{eqn:clustering_Kmeans_sdp}) does not require the knowledge of the cluster sizes other than the number of clusters $K$. Thus it can handle the general case for unequal cluster sizes.

\subsection{Rate of convergence of SDP for $K$-means in Hilbert spaces}
\label{subsec:rate_SDP_kernel_Kmeans}

Now we are in the position to state the rate of convergence for the SDP relaxation (\ref{eqn:clustering_Kmeans_sdp}) for the generalized $K$-means clustering. For simplicity, we assume that the trace norms of the covariance operators for the $K$-cluster distributions $P_{1},\dots,P_{K}$ are equal. If the trace norms are not all equal, then a similar de-biased SDP in \cite{BuneaGiraudRoyerVerzelen2016} can be considered. Denote the minimum cluster size as $\underline{n} = \min_{1 \leq k \leq K} n_{k}$. 

\begin{thm}[Exponential rate of convergence of SDP for generalized $K$-means]
\label{thm:rate_SDP_kernel_kmeans_general_case}
Let $X_{1},\dots,X_{n}$ be a sample of independent random variables in Hilbert space $\bH$ such that $X_{i} \sim P_{k}$ for $i \in G_{k}^{*}$. Let $\langle\cdot,\cdot\rangle_{\bH}$ and $\|\cdot\|_{\bH}$ be the associated inner product and Hilbert norm with $\bH$, and $\mu_{k}= \E X_{i}$, $\Sigma_{k} = \E[(X_{i}-\mu_k) \otimes (X_{i}-\mu_k)]$ be the covariance operator of $X_{i}, i \in G_{k}^{*}$. Suppose that $\bH$ is separable, and $X_{i} \sim \subg(\Sigma_{k})$ for $i \in G_{k}^{*}$ such that $\|X_{i}\|_{\psi_{2},\Sigma_{k}} \leq L$ and Assumption \ref{ass:bernstein_squared_norm} holds with $\Gamma_i = \Sigma_i$ therein being equal to $\Sigma_k$. In addition, assume $(\Sigma_{k})_{k=1}^{K}$ to be positive definite trace class, and $\|\Sigma_{1}\|_{\tr} = \cdots = \|\Sigma_{K}\|_{\tr}$. 
Define
\[
\SNR^{2} = {\Delta^{2} \over L^{2} \|\Sigma\|_{\op}} \wedge {\underline{n} \Delta^{4} \over L^{4} \|\Sigma\|_{\HS}^{2}}\quad\mbox{with } \Delta = \min_{1\leq i\neq j\leq K} \|\mu_i-\mu_j\|_{\bH}
\]
as the squared signal-to-noise ratio, and suppose $\Sigma \succeq \Sigma_{k}$ for all $k = 1,\dots,K$.
Then there exist universal constants $c_0, c'_0, c, C_{1}, C_{2} > 0$ such that as long as $\SNR^2 \geq  c_0\, n/\underline{n}$ and $\underline{n}^{2} K \geq c'_0 n$, it holds that
\begin{equation}
\label{eqn:rate_SDP_kernel_kmeans_general_case}
|\hat{Z}-Z^{*}|_{1} \leq C_1 \exp(-C_{2} \SNR^{2})
\end{equation}
with probability at least $1-c/n^{2}$.
\end{thm}
This theorem characterizes the hardness of clustering through the squared signal-to-noise ratio $\SNR^2$ that depends on the ratio of squared between-cluster separation rate $\Delta^2$ to within-clustering variation $L^2 \|\Sigma\|_{\op}$ or $L^2 \|\Sigma\|_{\HS}$.
We postpone its proof to Section~\ref{sec:proof_of_k_means}. It turns out that both terms in $\SNR^2$ are necessary depending on different regimes of parameters $\Delta$ and $\Sigma$.
For the optimality of the exponent $\SNR^2$ in the convergence rate for Euclidean space clustering, namely $\bH=\bR^p$, we refer to Section 3.3 of \cite{GiraudVerzelen2018} for a detailed discussion. In particular, if we instead use the weaker version of the concentration inequality~\eqref{lem:Hanson-Wright_ineq_Hilbert_space_linear_part_less_sharp}, then an extra $p$ factor will appear in the denominator of each term in $\SNR^2$, which is clearly suboptimal. 


Our proof is based on the inequality $\langle A, Z^\ast\rangle  \leq \langle A, \hat Z\rangle$, which is true due to the optimality of $\hat Z$ and the feasibility of $Z^\ast$. In particular, in the analysis of $\langle A, \hat Z - Z^\ast\rangle$ by decomposing the similarity matrix $A$ as a sum of its expectation and random fluctuations, one remainder term caused by the random fluctuations involves a quadratic form over Hilbert space $\bH$ as the $Q$ in Theorem~\ref{thm:HW_ineq_Hilbert_space}. In particular, we prove a uniform version of the Hanson-Wright inequality that leads to the exponential convergence rate~\eqref{eqn:rate_SDP_kernel_kmeans_general_case} in Theorem~\ref{thm:rate_SDP_kernel_kmeans_general_case} by combining our Theorem~\ref{thm:HW_ineq_Hilbert_space} with a careful union bound technique developed in \cite{FeiChen2018} that utilizes the geometric structure of $A$ and improves upon a naive union bound argument via covering.

Theorem~\ref{thm:rate_SDP_kernel_kmeans_general_case} provides a partial recovery bound for clustering. Next, we show that exact recovery can be achieved by properly rounding the SDP solution $\hat Z$.
More specifically, we consider the rounding algorithm that proceeds as follows: 1.~let $j_1=1$ and $\hat G_{1}$ be the set of all indices $i$ such that $\hat Z_{j_1i} \geq \frac{1}{2}\hat Z_{j_1j_1}$; 2.~let $j_2$ be the smallest index in $[n]\setminus \hat {G}_1$ and $\hat G_{2}$ be the set of all indices $i$ such that $\hat Z_{j_2i} \geq \frac{1}{2}\hat Z_{j_2j_2}$; \ldots, end until the remainder index set $[n]\setminus \bigcup_{k=1}^{\hat K} \hat G_k$ becomes empty for some $\hat K\geq 1$.
Thanks to Theorem \ref{thm:rate_SDP_kernel_kmeans_general_case}, exact recovery of the true clustering structure is an immediate consequence when $\SNR^2\gtrsim \max\{n/\underline{n},\,\log n \}$.

\begin{cor}[Exact recovery of SDP for generalized $K$-means]\label{coro:exact_recovery}
In the setting of Theorem \ref{thm:rate_SDP_kernel_kmeans_general_case}, suppose $\SNR^2\geq c_1 \max\{n/\underline{n},\,\log n \}$ and $\underline{n}^{2} K \geq c_2 n$ for some universal constants $c_1, c_2>0$, then
\[
\Prob(\hat K = K \mbox{ and } \hat{G}_{k}=G_{k}^{*}, \; \forall k=1,\dots,K) \geq 1 - C n^{-2}
\]
for some universal constant $C > 0$. 
\end{cor}

\subsection{Implications in functional data clustering}\label{sec:application_FDA}
In this subsection, we discuss the consequence of applying Theorem~\ref{thm:rate_SDP_kernel_kmeans_general_case} to Example~\ref{ex:functional_data}. For simplicity, we assume that for each $k=1,\ldots,K$, the sampling measure $P_k$ is a Gaussian process (GP) over Hilbert space $\bL^2[0,1]$ with inner product $\langle\cdot,\cdot\rangle_{\bL^2}$.  In particular, we use Theorem~\ref{thm:rate_SDP_kernel_kmeans_general_case} to study and compare the uses of different inner products (such as Sobolev inner products with different orders) in constructing the similarity matrix $A$ in the generalized $K$-means for functional data clustering.

Recall the definition of a Gaussian random variable in a Hilbert space in Definition~\ref{defin:gaussian_Hilbert_space}. When the Hilbert space is a function space, the law $N(\mu,\Sigma)$ of a GP is completely determined by its mean function $\mu:\,[0,1]\to\bR \in \bL^2[0,1]$ and covariance function $\Sigma:\,[0,1]^2\to \bL^2[0,1]$, where $\mu(t) = \E[X(t)]$ and $\Sigma(t,t')=\E[(X(t)-\mu(t))(X(t')-\mu(t'))]$ for any GP realization $X=\{X(t):\,t\in[0,1]\}$. The covariance function $\Sigma$ can be identified with the covariance operator through
\begin{align*}
\Sigma f (t) = \int_0^1\Sigma(t,t')\,f(t')\,dt',\quad\mbox{for all }f\in\bL^2[0,1] \mbox{ and }t\in[0,1].
\end{align*}
Suppose now we have another Hilbert space $\bH'\subset \bH$, such as the Sobolev space $\bS^k[0,1]$ for some $k\geq 1$, such that the second moment of $\|X-\mu\|_{\bH'}$ is still bounded relative to the stronger norm $\|\cdot\|_{\bH'}$ associated with $\bH'$, that is $\E [\|X-\mu\|_{\bH'}^2] < \infty$. This implies $X-\mu \in\bH'$ almost surely, and $\langle h,X-\mu\rangle_{\bH'}$ is Gaussian for all $h\in\bH'$. As a consequence, $X-\mu$ remains a Gaussian random variable in the new Hilbert space $\bH'$ \cite{van2008reproducing}, as long as $\E [\|X-\mu\|_{\bH'}^2] < \infty$. Here $\mu$ may or may not belong to $\bH'$ depending on whether $\|\mu\|_{\bH'}$ is finite or infinite. We use $\Sigma'$ to denote its covariance operator as a Gaussian random variable in $\bH'$. In cases where $\Sigma$ has rapid eigenvalue decay (polynomial or exponential), the operator and the Hilbert-Schmidt norms of $\Sigma$ and $\Sigma'$ will be dominated by their respective top eigenvalues, henceforth comparable in magnitudes. 

Returning to the functional data clustering, we assume $X_i\sim N(\mu_k,\Sigma_k)$ for $i\in G_k^\ast$ as Gaussian random variables in $\bH$.
Consider two choices $a_{ij} =  \langle X_i,X_j\rangle_{\bH}$ and $a'_{ij} = \langle X_i,X_j\rangle_{\bH'}$ for constructing the similarity matrix $A$ in the SDP for the generalized $K$-means clustering. From our previous discussion, we know that $X_i-\mu_k$ remains Gaussian in $\bH'$ as long as $\E [\|X_i-\mu_k\|_{\bH'}^2] < \infty$. We use $\Sigma'_k$ to denote the covariance operator of $X_i-\mu_k$ as a Gaussian random variable in $\bH'$. 
We can then apply Theorem~\ref{thm:rate_SDP_kernel_kmeans_general_case} with Hilbert space $\bH$ and $\bH'$ to obtain the signal-to-noise ratios under these two choices,
\begin{align*}
\SNR^{2} &= {\Delta^{2} \over L^{2} \|\Sigma\|_{\op}} \wedge {\underline{n} \Delta^{4} \over L^{4} \|\Sigma\|_{\HS}^{2}}\quad\mbox{with } \Delta = \min_{1\leq i\neq j\leq K} \|\mu_i-\mu_j\|_{\bH},\quad\mbox{and}\\
(\SNR')^{2} &= {(\Delta')^{2} \over L^{2} \|\Sigma'\|_{\op}} \wedge {\underline{n} (\Delta')^{4} \over L^{4} \|\Sigma'\|_{\HS}^{2}}\quad\mbox{with } \Delta' = \min_{1\leq i\neq j\leq K} \|\mu_i-\mu_j\|_{\bH'},
\end{align*}
where $\Sigma\succeq\Sigma_k$ and $\Sigma'\succeq \Sigma'_k$ for each $k$. The denominators of $\SNR^2$ and $(\SNR')^2$ are comparable when $\Sigma$ and $\Sigma'$ have rapid eigenvalue decay, while the signal strength $\Delta'$ can be much larger than $\Delta$, making the overall $(\SNR')^2$ larger as well. For functional data with $\bH=\bL^2[0,1]$, faster eigenvalue decay in the covariance operator corresponds to a higher smoothness order of the sample path. For example, if $\gamma_1\geq\gamma_2\geq\ldots$ are ordered eigenvalues of $\Sigma$ with $\gamma_j\approx j^{-2\beta-1}$ for $j=1,2,\ldots$ and some $\beta>0$, then sample paths from $N(0,\Sigma)$ are at least $\beta$ times differentiable \cite{rasmussen2004gaussian} almost surely. If we choose $\bH'$ to be $\bS^{k}[0,1]$ for any $0\leq k\leq\lfloor\beta\rfloor$, where $\lfloor\beta\rfloor$ denotes the largest integer smaller than $\beta$, then $\E [\|X_i-\mu_k\|_{\bH'}^2] < \infty$. On the other hand side, $\Delta'$ can be much larger than $\Delta$ when 
the difference $\{\mu_i-\mu_j:\,1\leq i\neq j\leq K\}$ has smoothness order (characterized via the decay rate of coefficients with respect to eigenfunctions $\{e_i\}$ of $\Sigma$) lower than $k$. In such scenarios, using the inner product induced by a stronger norm in constructing the similarity matrix $A$ may increase the signal-to-noise ratio and reduce the SDP error $|\hat Z-Z^\ast|_1$.

\section{Proof of main results}
\label{sec:proof_main_results}

\subsection{Proof of main results in Section \ref{subsec:hanson-wright_hilbert_space}}

In this subsection, we prove Proposition \ref{prop:Hanson-Wright_ineq_Hilbert_space_diagonal_free}, Theorem \ref{thm:Hanson-Wright_ineq_Hilbert_space_nonnegative_diagonal}, and \ref{thm:HW_ineq_Hilbert_space}.

\begin{proof}[Proof of Proposition \ref{prop:Hanson-Wright_ineq_Hilbert_space_diagonal_free}]
By Markov's inequality, we have for any $\lambda > 0$ and $t > 0$, 
\[
\Prob(S \geq t) \leq e^{-\lambda t} \E[e^{\lambda S}].
\]
{\bf Step 1: decoupling.} Let $\delta_{1},\dots,\delta_{n} \in \{0,1\}$ be i.i.d. symmetric Bernoulli random variables (i.e., $\Prob(\delta_{i} = 0) = \Prob(\delta_{i} = 1) = 1/2$) that are independent of $X_{1},\dots,X_{n}$. Since 
\[
\E[\delta_{i}(1-\delta_{j})] = \left\{
\begin{array}{cc}
0 & \text{if } i=j \\
1/4 & \text{if } i \neq j\\
\end{array}
\right. ,
\]
we have $S = 4\E_{\delta}[S_{\delta}]$, where $S_{\delta} = \sum_{i,j=1}^{n} \delta_{i}(1-\delta_{j}) a_{ij} \langle X_{i}, X_{j} \rangle$ and $\E_{\delta}[\cdot]$ is the expectation taken with respect to the random variables $\delta_{i}$. Below, $\E_{X}[\cdot]$ is similarly defined. By Jensen's inequality, we get 
\[
\E[e^{\lambda S}] \leq E_{X,\delta}[e^{4 \lambda S_{\delta}}].
\]
Let $\Lambda_{\delta} = \{i \in [n] : \delta_{i}=1\}$. Then we can write 
\[
S_{\delta} = \sum_{i \in \Lambda_{\delta}} \sum_{j \in \Lambda_{\delta}^{c}} a_{ij} \langle X_{i}, X_{j} \rangle = \sum_{j \in \Lambda_{\delta}^{c}} \langle \sum_{i \in \Lambda_{\delta}} a_{ij} X_{i}, X_{j} \rangle.
\]
Taking the expectation with respect to $(X_{j})_{j \in \Lambda_{\delta}^{c}}$ (i.e., conditioning on $(\delta_{i})_{i=1,\dots,n}$ and $(X_{i})_{i \in \Lambda_{\delta}}$), it follows from the assumption $X_{i}$ are independent $\subg(\Gamma)$ with mean zero that 
\begin{align*}
\E_{(X_{j})_{j \in \Lambda_{\delta}^{c}}}[ e^{4\lambda S_{\delta}} ] \leq& e^{8 \lambda^{2} \sigma_{\delta}^{2}},
\end{align*}
where $\sigma_{\delta}^{2} = \sum_{j \in \Lambda_{\delta}^{c}} L_{j}^{2} \langle \Gamma (\sum_{i \in \Lambda_{\delta}} a_{ij} X_{i}), (\sum_{i \in \Lambda_{\delta}} a_{ij} X_{i}) \rangle$. Thus we get 
\[
\E_{X}[ e^{4\lambda S_{\delta}} ] \leq \E_{X} \left[ e^{8 \lambda^{2} \sigma_{\delta}^{2}} \right].
\]
{\bf Step 2: reduction to Gaussian random variables.} For $j=1,\dots,n$, let $g_{j}$ be independent $N(0, 16 L_{j}^{2} \Gamma)$ random variables in $\bH$ that are independent of $X_{1},\dots,X_{n}$ and $\delta_{1},\dots,\delta_{n}$. Define 
\[
T := \sum_{j \in \Lambda_{\delta}^{c}} \langle g_{j}, \sum_{i \in \Lambda_{\delta}} a_{ij} X_{i} \rangle.
\]
Then, by the definition of Gaussian random variables in $\bH$, we have 
\begin{align*}
\E_{g}[e^{\lambda T}] =& \prod_{j \in \Lambda_{\delta}^{c}} \E_{g} \left[ e^{\langle g_{j}, \lambda \sum_{i \in \Lambda_{\delta}} a_{ij} X_{i} \rangle} \right] \\
=& \exp\left( 8 \lambda^{2} \sum_{j \in \Lambda_{\delta}^{c}} L_{j}^{2} \langle \Gamma (\sum_{i \in \Lambda_{\delta}} a_{ij} X_{i}), (\sum_{i \in \Lambda_{\delta}} a_{ij} X_{i}) \rangle \right) = \exp\left( 8 \lambda^{2} \sigma_{\delta}^{2} \right).
\end{align*}
So it follows that 
\[
\E_{X}[ e^{4\lambda S_{\delta}} ] \leq \E_{X,g}[e^{\lambda T}].
\]
Since $T = \sum_{i \in \Lambda_{\delta}} \langle \sum_{j \in \Lambda_{\delta}^{c}} a_{ij} g_{j}, X_{i} \rangle$, we have 
\[
\E_{(X_{i})_{i \in \Lambda_{\delta}}}[ e^{\lambda T} ] \leq \exp\left( {\lambda^{2} \over 2} \sum_{i \in \Lambda_{\delta}} L_{i}^{2} \langle \Gamma (\sum_{j \in \Lambda_{\delta}^{c}} a_{ij} g_{j}), (\sum_{j \in \Lambda_{\delta}^{c}} a_{ij} g_{j}) \rangle \right),
\]
which implies that 
\begin{equation}
\label{eqn:intermediate_step_HW}
\E_{X}[ e^{4\lambda S_{\delta}} ] \leq \E_{g} \left[ \exp\left( \lambda^{2} \tau_{\delta}^{2} / 2 \right) \right], 
\end{equation}
where $\tau_{\delta}^{2} = \sum_{i \in \Lambda_{\delta}} L_{i}^{2} \langle \Gamma (\sum_{j \in \Lambda_{\delta}^{c}} a_{ij} g_{j}), (\sum_{j \in \Lambda_{\delta}^{c}} a_{ij} g_{j}) \rangle$. 

\noindent {\bf Step 3: diagonalization.} Since $\Gamma \in \cB(\bH)$ is trace class (thus compact) and positive definite, it follows from Theorem 4.2.4 in \cite{HsingEubank2015_Wiley} that the eigendecomposition of $\Gamma$ is given by 
\[
\Gamma = \sum_{k=1}^{\infty} \gamma_{k} (e_{k} \otimes e_{k}),
\]
where $\gamma_{k} \geq 0$ are eigenvalues of $\Gamma$ and $(e_{k})_{k=1}^{\infty}$ are eigenfunctions forming a CONS of $\overline{\Image(\Gamma)}$; namely $\Gamma h = \sum_{k=1}^{\infty} \gamma_{k} \langle h, e_{k} \rangle e_{k}$ for every $h \in \bH$. Here, $\otimes$ denotes the tensor product and $\overline{\Image(\Gamma)}$ denotes the closure of the image of $\Gamma$. In addition, there exists a unique positive definite square root operator $\Gamma^{1/2} \in \cB(\bH)$ such that $\Gamma^{1/2} \Gamma^{1/2} = \Gamma$ (cf. Theorem 3.4.3 in \cite{HsingEubank2015_Wiley}). Then we have $\Gamma^{1/2} g_{j} = \sum_{k=1}^{\infty} \gamma_{k}^{1/2} \langle g_{j}, e_{k} \rangle e_{k}$ and 
\begin{align*}
\tau_{\delta}^{2} =& \sum_{i \in \Lambda_{\delta}} L_{i}^{2} \langle \Gamma^{1/2} (\sum_{j \in \Lambda_{\delta}^{c}} a_{ij} g_{j}), \Gamma^{1/2} (\sum_{j \in \Lambda_{\delta}^{c}} a_{ij} g_{j}) \rangle = \sum_{i \in \Lambda_{\delta}} L_{i}^{2} \| \Gamma^{1/2} (\sum_{j \in \Lambda_{\delta}^{c}} a_{ij} g_{j}) \|^{2} \\
& \quad = \sum_{i \in \Lambda_{\delta}} L_{i}^{2} \| \sum_{j \in \Lambda_{\delta}^{c}} a_{ij} \Gamma^{1/2} g_{j} \|^{2} = \sum_{i \in \Lambda_{\delta}} L_{i}^{2} \| \sum_{k=1}^{\infty} \gamma_{k}^{1/2} (\sum_{j \in \Lambda_{\delta}^{c}} a_{ij} \langle g_{j}, e_{k} \rangle ) e_{k} \|^{2} \\
& \quad = \sum_{k=1}^{\infty} \gamma_{k} \sum_{i \in \Lambda_{\delta}} \left( \sum_{j \in \Lambda_{\delta}^{c}} L_{i} a_{ij}  \langle g_{j}, e_{k} \rangle  \right)^{2},
\end{align*}
where the last step follows from Parseval's identity. Note that 
\[
\|\Gamma^{1/2} e_{k}\|^{2} = \langle \Gamma e_{k}, e_{k} \rangle = \langle \gamma_{k} e_{k}, e_{k} \rangle = \gamma_{k}.
\]
Thus for any $\lambda \in \bR$, 
\[
\E e^{\lambda \langle g_{j}, e_{k} \rangle} = e^{8 L_{j}^{2} \lambda^{2} \langle \Gamma e_{k}, e_{k} \rangle} = e^{8 L_{j}^{2} \lambda^{2} \|\Gamma^{1/2} e_{k}\|^{2}} = e^{8 L_{j}^{2} \lambda^{2} \gamma_{k}}, 
\]
which implies that $G_{jk} := \langle g_{j}, e_{k} \rangle,j=1,\dots,n$, are independent $N(0, 16 L_{j}^{2} \gamma_{k})$ random variables. Now let $f = (\sqrt{\gamma_{1}} f_{1}^{T}, \sqrt{\gamma_{2}} f_{2}^{T}, \dots)^{T}$, where $f_{k} = (G_{1k}, \dots, G_{nk})^{T}$ for $k=1,2,\dots$. Then $f \sim N(0, \widetilde{\Gamma})$, where $\widetilde{\Gamma} = (\widetilde{\Gamma}_{km})_{k,m=1}^{\infty}$ with $\widetilde{\Gamma}_{km} = \diag(E_{km,11}, \dots, E_{km,nn})$ and $E_{km,jj} = \sqrt{\gamma_{k}\gamma_{m}} \E[G_{jk}G_{jm}]$. Note that 
\begin{align*}
\E[G_{jk}G_{jm}] =& \E[\langle \langle g_{j}, e_{k} \rangle g_{j}, e_{m} \rangle] = \langle (\E \langle g_{j} \otimes g_{j} \rangle) e_{k}, e_{m} \rangle \\
=& 16 L_{j}^{2} \langle \Gamma e_{k}, e_{m} \rangle = 16 L_{j}^{2} \langle \gamma_{k} e_{k}, e_{m} \rangle = 16 L_{j}^{2} \gamma_{k} \vone(k=m). 
\end{align*}
Thus $\widetilde{\Gamma}_{km}$ is an $n \times n$ matrix of all zeros if $k \neq m$, and $\widetilde{\Gamma}_{kk} = 16 \gamma_{k}^{2} \diag(L_{1}^{2}, \dots, L_{n}^{2})$. 

\noindent {\bf Step 4: bound the eigenvalues.} Let $P_{\delta} : \bR^{n} \to \bR^{n}$ be the restriction matrix such that $P_{\delta,ii} = 1$ if $i \in \Lambda_{\delta}$ and $P_{\delta,ij} = 0$ otherwise. Let further $R_{\delta} = \diag(P_{\delta} \widetilde{A} (I_{n}-P_{\delta}), P_{\delta} \widetilde{A} (I_{n}-P_{\delta}), \dots)$ and $Z = (Z_{1}, Z_{2}, \dots)^{T}$, where $\widetilde{A} = (\widetilde{a}_{ij})_{i,j=1}^{n}$ with $\widetilde{a}_{ij} = L_{i} a_{ij}$ and $Z_{i}$ are i.i.d. standard Gaussian random variables in $\bR$. By the rotational invariance of Gaussian distributions, we have 
\[
\tau_{\delta}^{2} = \left\| R_{\delta} f \right\|^{2} \stackrel{d}{=} \left\| R_{\delta} \widetilde{\Gamma}^{1/2} Z \right\|^{2} = Z^{T} \widetilde{\Gamma}^{1/2} R_{\delta}^{T} R_{\delta} \widetilde{\Gamma}^{1/2} Z \stackrel{d}{=} \sum_{k=1}^{\infty} s_{k}^{2} Z_{k}^{2},
\]
where $(s_{k}^{2})_{k=1}^{\infty}$ are the eigenvalues of $\widetilde{\Gamma}^{1/2} R_{\delta}^{T} R_{\delta} \widetilde{\Gamma}^{1/2}$.  So it follows that 
\[
\max_{k} s_{k}^{2} \leq \| R_{\delta} \|_{\op}^{2} \| \widetilde{\Gamma} \|_{\op} \leq \| \widetilde{A} \|_{\op}^{2} \| \widetilde{\Gamma} \|_{\op} \leq L^{2} \| A \|_{\op}^{2} \| \widetilde{\Gamma} \|_{\op}, 
\]
where 
\[
\| \widetilde{\Gamma} \|_{\op} \leq 16 (\max_{1 \leq j \leq n} \|X_{j}\|_{\psi_{2}}^{2}) (\max_{k} \gamma_{k}^{2}) \leq 16 L^{2} \|\Gamma\|_{\op}^{2}.
\]
In addition, we also have 
\begin{align*}
\sum_{k} s_{k}^{2} =& \tr(\widetilde{\Gamma}^{1/2} R_{\delta}^{T} R_{\delta} \widetilde{\Gamma}^{1/2}) = \tr( R_{\delta} \widetilde{\Gamma} R_{\delta}^{T} ) = \sum_{k=1}^{\infty} \tr( [P_{\delta} \widetilde{A} (I_{n}-P_{\delta})] \widetilde{\Gamma}_{kk} [P_{\delta} \widetilde{A} (I_{n}-P_{\delta})]^{T}) \\
\leq& \sum_{k=1}^{\infty} 16 L^{2} \gamma_{k}^{2} \|P_{\delta} \widetilde{A} (I_{n}-P_{\delta})\|_{\HS}^{2} \leq \sum_{k=1}^{\infty} 16 L^{2} \gamma_{k}^{2} \|\widetilde{A}\|_{\HS}^{2} \leq 16 L^{4} \|\Gamma\|_{\HS}^{2} \|A\|_{\HS}^{2}.
\end{align*}
Invoking (\ref{eqn:intermediate_step_HW}), we get 
\[
\E_{X}[ e^{4\lambda S_{\delta}} ] \leq \prod_{k=1}^{\infty} \E_{Z} [ \exp(\lambda^{2} s_{k}^{2} Z_{k}^{2}/2) ].
\]
Since $Z_{k}^{2}$ are i.i.d. $\chi^{2}_{1}$ random variables with the moment generating function $\E[e^{t Z_{k}^{2}}] = (1-2t)^{-1/2}$ for $t < 1/2$, we have 
\[
\E_{X}[ e^{4\lambda S_{\delta}} ] \leq \prod_{k=1}^{\infty} {1 \over \sqrt{1-\lambda^{2} s_{k}^{2}}}, \qquad \text{if } \max_{k} \lambda^{2} s_{k}^{2} < 1.
\]
Using $(1-z)^{-1/2} \leq e^{z}$ for $z \in [0,1/2]$, we get that if $16 L^{4} \|A\|_{\op}^{2} \|\Gamma\|_{\op}^{2} \lambda^{2} < 1$, then 
\[
\E_{X}[ e^{4\lambda S_{\delta}} ] \leq \exp(\lambda^{2} \sum_{k=1}^{\infty} s_{k}^{2}) \leq \exp(16\lambda^{2} L^{4} \|\Gamma\|_{\HS}^{2} \|A\|_{\HS}^{2}).
\]
Note that the last inequality is uniform in $\delta$. Taking expectation with respect to $\delta$, we obtain that 
\[
\E_{X} [e^{\lambda S}] \leq \E_{X,\delta}[e^{4 \lambda S_{\delta}}] \leq \exp( 16\lambda^{2} L^{4} \|\Gamma\|_{\HS}^{2} \|A\|_{\HS}^{2} ),
\]
whenever $0 < \lambda < (4 L^{2} \|A\|_{\op} \|\Gamma\|_{\op})^{-1}$. 

\noindent {\bf Step 5: conclusion.} Now we have 
\[
\Prob(S \geq t) \leq \exp(-\lambda t + 16 \lambda^{2} L^{4} \|\Gamma\|_{\HS}^{2} \|A\|_{\HS}^{2}) \quad \text{for } 0 < \lambda \leq (8 L^{2} \|A\|_{\op} \|\Gamma\|_{\op})^{-1}.
\]
Optimizing in $\lambda$, we deduce that there exists a universal constant $C > 0$ such that 
\[
\Prob(S \geq t) \leq \exp \left[ -C \min \left( {t^{2} \over L^{4} \|\Gamma\|_{\HS}^{2} \|A\|_{\HS}^{2} } , {t \over L^{2} \|\Gamma\|_{\op} \|A\|_{\op} } \right) \right], 
\]
as desired in (\ref{eqn:Hanson-Wright_ineq_Hilbert_space_diagonal_free}).
\end{proof}

\begin{proof}[Proof of Theorem \ref{thm:Hanson-Wright_ineq_Hilbert_space_nonnegative_diagonal}]
Decompose $Q = \sum_{i=1}^{n} a_{ii} \|X_{i}\|^{2} + S$, where $S = \sum_{1 \leq i \neq j \leq n} a_{ij} \langle X_{i}, X_{j} \rangle$. In view of the off-diagonal sum bound for $S$ in Proposition \ref{prop:Hanson-Wright_ineq_Hilbert_space_diagonal_free}, it suffices to show the following inequality for the diagonal sum: for any $t > 0$, 
\begin{equation}
\label{eqn:Hanson-Wright_ineq_Hilbert_space_nonnegative_diagonal_diagonal_sum} 
\begin{split}
& \Prob \left( \sum_{i=1}^{n} a_{ii} \|X_{i}\|^{2} \geq \sum_{i=1}^{n} a_{ii} L_{i}^{2} \|\Gamma\|_{\tr} + t \right) \\
& \qquad \leq \exp \left[ -C \min \left( {t^{2} \over L^{4} \|\Gamma\|_{\HS}^{2} \sum_{i=1}^{n} a_{ii}^{2} }, {t \over L^{2} \|\Gamma\|_{\op} \max_{1 \leq i \leq n} a_{ii}} \right) \right], 
\end{split}
\end{equation}
since $\sum_{i=1}^{n} a_{ii}^{2} \leq \|A\|_{\HS}^{2}$ and $\overline{a} := \max_{1 \leq i \leq n} a_{ii} \leq \|A\|_{\op}$. By Markov's inequality and Lemma \ref{lem:mgf_bound_subgaussian_hilbert_space}, we have for any $\lambda > 0$ and $t > 0$, 
\begin{align*}
& \Prob \left( \sum_{i=1}^{n} a_{ii} (\|X_{i}\|^{2} - L_{i}^{2} \|\Gamma\|_{\tr} ) \geq t \right) \leq e^{-\lambda t} \prod_{i=1}^{n} \E [ e^{\lambda a_{ii} (\|X_{i}\|^{2}-L_{i}^{2} \|\Gamma\|_{\tr})} ] \\
& \qquad \leq e^{-\lambda t} \prod_{i=1}^{n} e^{2 \lambda^{2} a_{ii}^{2} L_{i}^{4} \|\Gamma\|_{\HS}^{2}} \leq \exp \left( -\lambda t + 2 \lambda^{2} (\sum_{i=1}^{n} a_{ii}^{2}) L^{4} \|\Gamma\|_{\HS}^{2} \right) 
\end{align*}
holds for all $0 \leq \lambda < (4 L^{2} \|\Gamma\|_{\op} \overline{a})^{-1}$. Choosing 
\[
\lambda = {t \over 4 (\sum_{i=1}^{n} a_{ii}^{2}) L^{4} \|\Gamma\|_{\HS}^{2}} \wedge {1 \over 8 \overline{a} L^{2} \|\Gamma\|_{\op}}, 
\]
we get (\ref{eqn:Hanson-Wright_ineq_Hilbert_space_nonnegative_diagonal_diagonal_sum}). 
\end{proof}

\begin{proof}[Proof of Theorem \ref{thm:HW_ineq_Hilbert_space}]
Under Assumption \ref{ass:bernstein_squared_norm}, we have the following standard moment generating function bound 
\[
\E \left[ e^{\lambda (\|X_{i}\|^{2} - \E \|X_{i}\|^{2})} \right] \leq e^{C \lambda^{2} \|\Gamma\|_{\HS}^{2} \over 2} \quad \forall |\lambda| < {1 \over 2 \|\Gamma\|_{\op}}. 
\]
See for example Chapter 2 in \cite{Wainwright2019_HDS}. Then we have for any $\lambda > 0$ and $t > 0$, 
\[
\Prob \left( \sum_{i=1}^{n} a_{ii} (\|X_{i}\|^{2} - \E \|X_{i}\|^{2}) \geq  t \right) \leq \exp \left( -\lambda t + C \lambda^{2} (\sum_{i=1}^{n} a_{ii}^{2}) \|\Gamma\|_{\HS}^{2} \right) \quad \forall |\lambda| < {1 \over 2 \overline{a} \|\Gamma\|_{\op}},  
\]
where $\overline{a} := \max_{1 \leq i \leq n} |a_{ii}|$. Note that $\sum_{i=1}^{n} a_{ii}^{2} \leq \|A\|_{\HS}^{2}$ and $\overline{a} \leq \|A\|_{\op}$. Optimizing over $\lambda$ and combining with Proposition \ref{prop:Hanson-Wright_ineq_Hilbert_space_diagonal_free}, we get 
\[
\Prob \left( Q - \E[Q] \geq t \right) \leq 2 \exp\left[ -C \min\left( {t^{2} \over L^{4} \|\Gamma\|_{\HS}^{2} \|A\|_{\HS}^{2}}, {t \over L^{2} \|\Gamma\|_{\op} \|A\|_{\op}} \right)\right].
\]
Applying the same argument by replacing $Q$ with $-Q$, we obtain (\ref{eqn:HW_ineq_Hilbert_space}) with constant 4, which can be reduced to 2 by adjusting the value of constant $C$. 
\end{proof}

\subsection{Proof of main results in Section \ref{sec:kernel_Kmeans}}\label{sec:proof_of_k_means}
In this subsection, we prove Theorem \ref{thm:rate_SDP_kernel_kmeans_general_case} and Corollary~\ref{coro:exact_recovery}.

\begin{proof}[Theorem \ref{thm:rate_SDP_kernel_kmeans_general_case}]
Recall that $\sC = \{Z_{n \times n} : Z^{T} = Z, Z \succeq 0, \tr(Z) = K, Z \vone_{n} = \vone_{n}, Z \geq 0\}$ is the SDP constraint set for the generalized $K$-means in (\ref{eqn:clustering_Kmeans_sdp}). For $i \in G_{k}^{*}$, let $\mu_{k} = \E[X_{i}]$ and $\delta_{i} = X_{i} - \mu_{k}$. For notation simplicity, we will omit in the proof the subscript $\bH$ in the Hilbert space inner product $\langle\cdot,\cdot\rangle_{\bH}$ and norm $\|\cdot\|_{\bH}$.

{\bf Step 1: a generic bound.} For any $Z \in \sC$, consider $\langle A, Z-Z^{*} \rangle = \sum_{i,j=1}^{n} a_{ij} (Z_{ij}-Z_{ij}^{*})$. Note that if $i \in G_{k}^{*}$ and $j \in G_{m}^{*}$, then 
\begin{align*}
a_{ij} =& \langle \mu_{k} + \delta_{i}, \mu_{m} + \delta_{j} \rangle = \langle \mu_{k}, \mu_{m} \rangle + \langle \mu_{k}, \delta_{j} \rangle + \langle \delta_{i}, \mu_{m} \rangle + \langle \delta_{i}, \delta_{j} \rangle \\
=& \langle \mu_{k}, \mu_{m} \rangle + \langle \mu_{k}-\mu_{m}, \delta_{j}-\delta_{i} \rangle + \langle \mu_{k}, \delta_{i} \rangle + \langle \delta_{j}, \mu_{m} \rangle + \langle \delta_{i}, \delta_{j} \rangle \\
=& -{1 \over 2} \|\mu_{k}-\mu_{m}\|^{2} + {1 \over 2} (\|\mu_{k}\|^{2} + \|\mu_{m}\|^{2}) + \langle \mu_{k}-\mu_{m}, \delta_{j}-\delta_{i} \rangle + \langle \mu_{k}, \delta_{i} \rangle + \langle \delta_{j}, \mu_{m} \rangle + \langle \delta_{i}, \delta_{j} \rangle.
\end{align*}
Since $\sum_{j=1}^{n} Z_{ij} = (Z \vone_{n})_{i} = 1$ for all $Z \in \sC$ and $Z^{*}$ is feasible for $\sC$, we have 
\[
\sum_{i,j=1}^{n} \sum_{k,m=1}^{K} \|\mu_{k}\|^{2} \vone(i \in G_{k}^{*}, j \in G_{m}^{*}) (Z_{ij}-Z_{ij}^{*}) = \sum_{i=1}^{n} \sum_{k=1}^{K} \|\mu_{k}\|^{2} \vone(i \in G_{k}^{*}) \sum_{j=1}^{n} (Z_{ij}-Z_{ij}^{*}) = 0
\]
and 
\[
\sum_{i,j=1}^{n} \sum_{k,m=1}^{K} \langle \mu_{k}, \delta_{i} \rangle \vone(i \in G_{k}^{*}, j \in G_{m}^{*}) (Z_{ij}-Z_{ij}^{*}) = \sum_{i=1}^{n} \sum_{k=1}^{K} \langle \mu_{k}, \delta_{i} \rangle \vone(i \in G_{k}^{*}) \sum_{j=1}^{n} (Z_{ij}-Z_{ij}^{*}) = 0.
\]
Then by the symmetry of $Z$ (i.e., $Z^{T}=Z$), we have 
\[
\langle A, Z-Z^{*} \rangle = \langle T_{1}+T_{2}+T_{3}+T_{4}, Z-Z^{*} \rangle,
\]
where for $i \in G_{k}^{*}$ and $j \in G_{m}^{*}$, 
\begin{align*}
T_{1,ij} = -{1 \over 2} \|\mu_{k}-\mu_{m}\|^{2}, & \qquad T_{2,ij} = \langle \mu_{k}-\mu_{m}, \delta_{j}-\delta_{i} \rangle, \\
T_{3,ij} = \langle \delta_{i}, \delta_{j} \rangle - \E \langle \delta_{i}, \delta_{j} \rangle, & \qquad T_{4,ij} = \E \langle \delta_{i}, \delta_{j} \rangle.
\end{align*}
Observe that 
\begin{align}\label{eqn:Lowbound}
\langle T_{1}, Z-Z^{*} \rangle =& -{1 \over 2} \sum_{1 \leq k \neq m \leq K} \|\mu_{k}-\mu_{m}\|^{2} \sum_{i \in G_{k}^{*}, j \in G_{m}^{*}} (Z_{ij}-Z^{*}_{ij}) \\
\label{eqn:Lowbound2}
=& -{1 \over 2} \sum_{1 \leq k \neq m \leq K} \|\mu_{k}-\mu_{m}\|^{2} |Z_{G_{k}^{*}G_{m}^{*}}|_{1},
\end{align}
where the last step follows from $Z \geq 0$ and $Z^{*}_{ij} = 0$ if $i \in G_{k}^{*}, j \in G_{m}^{*}$ for $k \neq m$. Here, $|Z_{G_{k}^{*}G_{m}^{*}}|_{1}=\sum_{i \in G_{k}^{*}, j \in G_{m}^{*}} |Z_{ij}|$. By definition, we have $\langle A, Z^{*} \rangle \leq \langle A, \hat{Z} \rangle$, which implies that $0 \leq \langle A, \hat{Z}-Z^{*} \rangle$. Thus we have 
\begin{equation}
\label{eqn:rate_SDP_kernel_kmeans_general_case_primative_step1.1}
0 \leq {1 \over 2} \sum_{1 \leq k \neq m \leq K} \|\mu_{k}-\mu_{m}\|^{2} |\hat{Z}_{G_{k}^{*}G_{m}^{*}}|_{1}  = \langle T_{1}, Z^{*} - \hat{Z} \rangle \leq \langle T_{2} + T_{3} +T_{4}, \hat{Z}-Z^{*} \rangle.
\end{equation}
Let $\Delta = \min_{1 \leq k \neq m \leq K} \|\mu_{k}-\mu_{m}\|$. By (\ref{eqn:ineq_3_feasible_set}) and (\ref{eqn:ineq_1_feasible_set}) in Lemma \ref{lem:some_ineq_feasible_set}, we have 
\[
|\hat{Z}-Z^{*}|_{1} \leq {2 n \over \underline{n}} |Z^{*}-Z^{*}\hat{Z}|_{1} = {4 n \over \underline{n}} \sum_{1 \leq k \neq m \leq K} |\hat{Z}_{G_{k}^{*}G_{m}^{*}}|_{1},
\]
where $\underline{n} = \min_{1 \leq k \leq K} n_{k}$. Then we get 
\begin{equation}
\label{eqn:rate_SDP_kernel_kmeans_general_case_primative}
|\hat{Z}-Z^{*}|_{1} \leq {8 n \over \Delta^{2} \underline{n}} \langle T_{2}+T_{3}+T_{4}, \hat{Z}-Z^{*} \rangle.
\end{equation}

{\bf Step 2: bound $\langle T_{4}, \hat{Z}-Z^{*} \rangle$.} Since $\delta_{1},\dots,\delta_{n}$ are independent with mean zero, we have 
\[
\langle T_{4}, \hat{Z}-Z^{*} \rangle = \sum_{i=1}^{n} \E \|\delta_{i}\|^{2} (\hat{Z}_{ii}-Z^{*}_{ii}).
\]
Since $\E \|\delta_{i}\|^{2} = \|\E[\delta_{i} \otimes \delta_{i}]\|_{\tr} = \|\Sigma_{k}\|_{\tr}$ if $i \in G_{k}^{*}$, and $\|\Sigma_{k}\|_{\tr}, k = 1,\dots,K$ are all equal, it follows that 
\[
\langle T_{4}, \hat{Z}-Z^{*} \rangle = \|\Sigma_{1}\|_{\tr} \tr(\hat{Z}-Z^{*}) = 0,
\]
where the last step is due to $\tr(\hat{Z}) = \tr(Z^{*}) = K$ since both $\hat{Z}, Z^{*} \in \sC$. 

{\bf Step 3: bound $\langle T_{2}, \hat{Z}-Z^{*} \rangle$.} Consider 
\begin{align*}
\langle T_{2}, \hat{Z}-Z^{*} \rangle =& \sum_{1 \leq k \neq m \leq K} \sum_{i,j=1}^{n} \langle \mu_{k}-\mu_{m}, \delta_{j}-\delta_{i} \rangle \vone(i \in G_{k}^{*}, j \in G_{m}^{*}) (\hat{Z}_{ij}-Z_{ij}^{*}) \\
=& \sum_{1 \leq k \neq m \leq K} \sum_{i \in G_{k}^{*}, j \in G_{m}^{*}} \langle \mu_{k}-\mu_{m}, \delta_{j}-\delta_{i} \rangle \hat{Z}_{ij} \\
=& 2 \sum_{1 \leq k \neq m \leq K} \sum_{i \in G_{k}^{*}, j \in G_{m}^{*}} \langle \mu_{k}-\mu_{m}, \delta_{i} \rangle \hat{Z}_{ij} \\
=& 2 \sum_{1 \leq k \neq m \leq K} \sum_{i \in G_{k}^{*}} \langle \mu_{k}-\mu_{m}, \delta_{i} \rangle |\hat{Z}_{iG_{m}^{*}}|_{1},
\end{align*}
where the third equality is due to symmetry. For each $k \neq m$, let $\epsilon_{i}^{(k,m)} = \langle \mu_{k}-\mu_{m}, \delta_{i} \rangle$ and $s_{k,m} = \sum_{i \in G_{k}^{*}} |\hat{Z}_{iG_{m}^{*}}|_{1}$. Since $|\hat{Z}_{iG_{m}^{*}}|_{1} \leq 1$, by Lemma \ref{lem:ordered_sum_inequality}, 
\[
\sum_{i \in G_{k}^{*}} \langle \mu_{k}-\mu_{m}, \delta_{i} \rangle |\hat{Z}_{iG_{m}^{*}}|_{1} \leq \sum_{i=1}^{s_{k,m}} \epsilon_{(i)}^{(k,m)},
\]
where $\epsilon_{(1)}^{(k,m)} \geq \dots \geq \epsilon_{(n)}^{(k,m)}$ are the order statistics of $\epsilon_{1}^{(k,m)},\ldots,\epsilon_{n}^{(k,m)}$. Note that $(\epsilon_{i}^{(k,m)})_{i=1}^{n}$ are i.i.d. mean-zero sub-gaussian random variables in $\bR$ with respect to $\tau_{k,m}^{2} := L^2\,\langle \Sigma (\mu_{k}-\mu_{m}), \mu_{k}-\mu_{m} \rangle$ (recall that $\Sigma \succeq \Sigma_{k}$ for all $k=1,\dots,K$). Thus for any $s=1,\dots,n$, we have $\sum_{i=1}^{s} \epsilon_{i}^{(k,m)}$ is a mean-zero sub-gaussian random variable with respect to $s \tau_{k,m}^{2}$. By the union bound, we get for all $t > 0$, 
\[
\Prob\left( \sum_{i=1}^{s} \epsilon_{(i)}^{(k,m)} \geq t \right) \leq {n \choose s} \exp\left( -{t^{2} \over 2 s \tau_{k,m}^{2}} \right) \leq \left( {e n \over s} \right)^{s} \exp\left( -{t^{2} \over 2 s \tau_{k,m}^{2}} \right).
\]
Now it follows that 
\begin{align*}
& \Prob \left( \exists 1 \leq k \neq m \leq K \text{ such that } \sum_{i=1}^{s_{k,m}} \epsilon_{(i)}^{(k,m)} \geq C_{1} \tau_{k,m} s_{k,m} \sqrt{\log\left( {n K \over s_{k,m} } \right)} \right) \\
\leq & \sum_{1 \leq k \neq m \leq K} \sum_{1 \leq s \leq n} \Prob\left( \sum_{i=1}^{s} \epsilon_{(i)}^{(k,m)} \geq C_{1} \tau_{k,m} s \sqrt{ \log \left( {n K \over s} \right) } \right) \\
\leq & \sum_{1 \leq k \neq m \leq K} \sum_{s=1}^{n}  \left( {e n \over s} \right)^{s} \exp\left( -{C_{1}^{2} \over 2} s \log \left( {nK \over s} \right) \right) \\
\leq & K^{2} \sum_{s=1}^{n}   \exp \left( -C_{2} s \log \left( {nK \over s} \right) \right) \leq {C_{3} K^{2} \over (nK)^{2} } = {C_{3} \over n^{2}}.
\end{align*}
Thus we have $\Prob(\cG_{1}) \geq 1-C_{3}n^{-2}$, where 
\[
\cG_{1} = \left\{ \sum_{i=1}^{s_{k,m}} \epsilon_{(i)}^{(k,m)} \leq C_{1} \tau_{k,m} s_{k,m} \sqrt{\log \left( {n K \over s_{k,m}} \right)} \quad \forall 1 \leq k \neq m \leq K \right\}. 
\]
By the Cauchy-Schwarz inequality,   
\begin{align*}
\langle T_{2}, \hat{Z}-Z^{*} \rangle \leq& 2C_{1} \sum_{1 \leq k \neq m \leq K} \tau_{k,m} s_{k,m} \sqrt{\log\left( {n K \over s_{k,m}} \right)} & \\
\leq& 2C_{1} \sqrt{\sum_{1 \leq k \neq m \leq K} \tau_{k,m}^{2} s_{k,m}} \sqrt{\sum_{1 \leq k \neq m \leq K} s_{k,m} \log \left( {n K \over s_{k,m}} \right)}
\end{align*}
on the event $\cG_{1}$. Since $s_{k,m} = |\hat{Z}_{G_{k}^{*} G_{m}^{*}}|_{1}$ and 
\[
\tau_{k,m} \leq L\, \|\Sigma^{1/2} (\mu_{k}-\mu_{m})\| \leq L\,\|\Sigma^{1/2}\|_{\op} \|\mu_{k}-\mu_{m}\| =L\, \|\Sigma\|_{\op}^{1/2} \|\mu_{k}-\mu_{m}\|, 
\]
it follows from the first equality in (\ref{eqn:rate_SDP_kernel_kmeans_general_case_primative_step1.1}) that 
\[
\sum_{1 \leq k \neq m \leq K} \tau_{k,m}^{2} s_{k,m} \leq \sum_{1 \leq k \neq m \leq K} L^2\,\|\Sigma\|_{\op} \|\mu_{k}-\mu_{m}\|^{2} |\hat{Z}_{G_{k}^{*} G_{m}^{*}}|_{1} = 2 L^2\,\|\Sigma\|_{\op} \langle T_{1}, Z^{*}-\hat{Z} \rangle.
\]
By (\ref{eqn:ineq_1_feasible_set}) in Lemma \ref{lem:some_ineq_feasible_set}, $S := |Z^{*}(\hat{Z}-Z^{*})|_{1} = 2 \sum_{1 \leq k \neq m \leq K} s_{k,m}$. Then it follows from Jensen's inequality that 
\[
\sum_{1 \leq k \neq m \leq K} s_{k,m} \log \left( {n K \over s_{k,m}} \right) \leq {S \over 2} \log \left( {2 n K^{3} \over S} \right).
\]
Thus we get 
\begin{equation}
\label{eqn:rate_SDP_kernel_kmeans_general_case_term_T2}
\langle T_{2}, \hat{Z}-Z^{*} \rangle \leq 2 C_{1} \,L \,\sqrt{\|\Sigma\|_{\op} \,\langle T_{1}, Z^{*}-\hat{Z} \rangle} \sqrt{S\log\left({2nK^{3} \over S}\right)}
\end{equation}
on the event $\cG_{1}$.

{\bf Step 4: bound $\langle T_{3}, \hat{Z}-Z^{*} \rangle$.} Decompose 
\[
\langle T_{3}, \hat{Z}-Z^{*} \rangle = \langle (I-Z^{*})T_{3}(I-Z^{*}), \hat{Z}-Z^{*} \rangle + \langle Z^{*}T_{3}, \hat{Z}-Z^{*} \rangle + \langle T_{3} Z^{*}, \hat{Z}-Z^{*} \rangle - \langle Z^{*}T_{3}Z^{*}, \hat{Z}-Z^{*} \rangle. 
\]
Note that 
\begin{align*}
\langle (I-Z^{*})T_{3}(I-Z^{*}), \hat{Z}-Z^{*} \rangle =_{(1)}& \langle T_{3}, (I-Z^{*})(\hat{Z}-Z^{*})(I-Z^{*}) \rangle \\
=_{(2)}& \langle T_{3}, (I-Z^{*})\hat{Z}(I-Z^{*}) \rangle \\
\leq_{(3)}& \|T_{3}\|_{\op} \|(I-Z^{*})\hat{Z}(I-Z^{*})\|_{\tr} \\
\leq_{(4)}& \|T_{3}\|_{\op} {|Z^{*}-Z^{*}\hat{Z}|_{1} \over 2\underline{n}}, 
\end{align*}
where $(1)$ follows from the symmetry of $Z^{*}$, $(2)$ from the idempotence of $Z^{*}$ (recall that $Z^{*}$ is a projection matrix such that $Z^{*} Z^{*} = Z^{*}$), $(3)$ from the duality of the operator and trace norms, and $(4)$ from (\ref{eqn:ineq_2_feasible_set}) in Lemma \ref{lem:some_ineq_feasible_set}. Let $\bS^{n-1}$ be the (compact) unit sphere in $\bR^{n}$ and $\cN$ be a $1/4$-net for $\bS^{n-1}$. By Lemma 5.2 and 5.4 in \cite{Vershynin2012_nonasymptotic}, we have $|\cN| \leq 9^{n}$ and $\|T_{3}\|_{\op} \leq 2 \max_{x \in \cN} x^{T} T_{3} x$. Thus, by the union bound, we have for any $t > 0$, 
\begin{equation}
\label{eqn:rate_SDP_kernel_kmeans_general_case_step4_union_bound}
\Prob(\|T_{3}\|_{\op} \geq t) \leq \sum_{x \in \cN} \Prob(x^{T} T_{3} x \geq t/2).
\end{equation}
Fix an $x \in \cN$. Note that $\|x x^{T}\|_{\HS}^{2} = \|x\|_{2}^{4} = 1$ and $\|x x^{T}\|_{\op} \leq 1$. Since $\Sigma \succeq \Sigma_{k}$ for all $k = 1,\dots,K$, we have $\delta_{i} \sim \subg(\Sigma)$ such that $\E[\delta_{i}] = 0$ and $\|\delta_{i}\|_{\psi_{2},\Sigma} \leq L$. By Theorem \ref{thm:HW_ineq_Hilbert_space} with $A = x x^{T}$, we get for all $t > 0$, 
\[
\Prob(x^{T} T_{3} x \geq t/2) = \Prob(\sum_{i,j=1}^{n} x_{i} x_{j} T_{3,ij} \geq t/2) \leq 2 \exp\left[ -C \min\left( {t^{2} \over L^{4} \|\Sigma\|_{\HS}^{2}}, {t \over L^{2} \|\Sigma\|_{\op}} \right)\right].
\]
Combining the last inequality with (\ref{eqn:rate_SDP_kernel_kmeans_general_case_step4_union_bound}), we obtain that with probability at least $1- c n^{-2}$, 
\[
\|T_3\|_{\op} \leq C_{5} L^{2} (\sqrt{n} \|\Sigma\|_{\HS} + n \|\Sigma\|_{\op}). 
\]
Then, 
\begin{align*}
\langle (I-Z^{*})T_{3}(I-Z^{*}), \hat{Z}-Z^{*} \rangle & \leq C_{5} L^{2} {\sqrt{n} \|\Sigma\|_{\HS} + n \|\Sigma\|_{\op} \over 2\underline{n}} |Z^{*}-Z^{*}\hat{Z}|_{1} \\
& \leq_{(1)}  C_{5} {\Delta^{2} \over 2} (c_{0}^{-1} + c_{0}^{-1/2}) |Z^{*} - Z^{*} \hat{Z}|_{1} \\
& =_{(2)}  C_{5} {\Delta^{2} \over 2} (c_{0}^{-1} + c_{0}^{-1/2}) 2 \sum_{1 \leq k \neq m \leq K} |\hat{Z}_{G_{k}^{*}G_{m}^{*}}|_{1} \\
& \leq_{(3)}  2 C_{5} (c_{0}^{-1} + c_{0}^{-1/2}) \langle T_{1}, Z^{*}-\hat{Z} \rangle \\
& \leq_{(4)}  \frac{1}{2}\,\langle T_{1},  Z^{*} -\hat{Z} \rangle,
\end{align*}
where $(1)$ follows from the definition of $\SNR^2$ and the condition that $\SNR^2 \geq  c_0\, n/\underline{n}$, $(2)$ from~\eqref{eqn:ineq_1_feasible_set} in Lemma~\ref{lem:some_ineq_feasible_set}, $(3)$ from the definition of $\Delta^{2}$ and~\eqref{eqn:rate_SDP_kernel_kmeans_general_case_primative_step1.1}, and $(4)$ from choosing $c_0$ sufficiently large.

Next, we consider $\langle Z^{*}T_{3}, \hat{Z}-Z^{*} \rangle = \langle Z^{*}T_{3}, Z^{*} \hat{Z}-Z^{*} \rangle$. By (\ref{eqn:Kmeans_true_membership_matrix}), we have 
\begin{align*}
& \langle Z^{*}T_{3}, Z^{*} \hat{Z}-Z^{*} \rangle = \sum_{i,j=1}^{n} (Z^{*} T_{3})_{ij} (Z^{*} \hat{Z}-Z^{*})_{ij} \\
& \qquad = \sum_{k,m=1}^{K} \sum_{i \in G_{k}^{*}} \sum_{j \in G_{m}^{*}} \left( \sum_{\ell=1}^n Z_{i\ell}^{*} T_{3,\ell j} \right) \left( \sum_{\ell=1}^{n} Z_{i\ell}^{*} \hat{Z}_{\ell j} - Z_{ij}^{*} \right) \\
& \qquad = \sum_{k,m=1}^{K} \sum_{i \in G_{k}^{*}} \sum_{j \in G_{m}^{*}} \left( {1 \over n_{k}} \sum_{\ell \in G_{k}^{*}} T_{3,\ell j} \right) {1 \over n_{k}} \left( \sum_{\ell \in G_{k}^{*}} \hat{Z}_{\ell j} - \vone(k=m)\right) \\
& \qquad = \sum_{k,m=1}^{K} \sum_{j \in G_{m}^{*}} \underbrace{\left( {(-1)^{\vone(k = m)} \over n_{k}} \sum_{\ell \in G_{k}^{*}} T_{3,\ell j} \right)}_{=:B_{kj}} \underbrace{\left| (Z^{*}-Z^{*} \hat{Z})_{G_{k}^{*}\, j} \right|_{1}}_{=:\beta_{kj}}. 
\end{align*}
Note that $\beta_{kj} \in [0,1]$. By Lemma \ref{lem:ordered_sum_inequality}, we have 
\[
\langle Z^{*}T_{3}, Z^{*} \hat{Z}-Z^{*} \rangle \leq \sum_{k,m=1}^{K} \sum_{j=1}^{b_{km}} B^{(k,m)}_{(j)},
\]
where $b_{km} = \sum_{j \in G_{m}^{*}} \beta_{kj} = | (Z^{*}-Z^{*} \hat{Z})_{G_{k}^{*} G_{m}^{*}} |_{1}$ and $B^{(k,m)}_{(1)} \geq B^{(k,m)}_{(2)} \geq \cdots$ is the ordered sequence of $(B_{kj})_{j \in G_{m}^{*}}$. Now fix a $(k,m)$. For any $E \subset G_{m}^{*}$ with $1 \leq q:=|E| \leq n_{m}$ , we can write 
\[
\sum_{j \in E} B_{kj} = \sum_{j,\ell=1}^{n} d^{(k,m)}_{\ell j} \left( \langle \delta_{\ell}, \delta_{j} \rangle - \E \langle \delta_{\ell}, \delta_{j} \rangle \right), 
\]
where $D^{(k,m)} = (d^{(k,m)}_{\ell j})_{\ell,j=1}^{n}$ and $d^{(k,m)}_{\ell j} = - n_{m}^{-1} \vone(j \in E) \vone(\ell \in G_{k}^{*})$. By Theorem \ref{thm:HW_ineq_Hilbert_space} (one-sided version) and the union bound, we have $t > 0$, 
\begin{align*}
\Prob \left( \sum_{j=1}^{q} B^{(k,m)}_{(j)} \geq t \right) \leq {n_{m} \choose q} \exp \left[ -C \min \left( {t^{2} \over L^{4} \|\Sigma\|_{\HS}^{2} \|D^{(k,m)}\|_{\HS}^{2}}, {t \over L^{2} \|\Sigma\|_{\op} \|D^{(k,m)}\|_{\op}} \right) \right].
\end{align*}
Since $\|D^{(k,m)}\|_{\HS} = \|D^{(k,m)}\|_{\op} = \sqrt{q/n_{m}}$, we deduce that 
\begin{align*}
& \Prob \bigg( \exists 1 \leq k,\,m \leq K \text{ such that } \sum_{i=1}^{b_{km}} B^{(k,m)}_{(j)} \geq \\
&\qquad\qquad \qquad\qquad\qquad \qquad \qquad C_{6}\,L^2\bigg(\|\Sigma\|_{\HS}\frac{b_{km} }{\sqrt{n_m}}   \sqrt{\log {n_m K \over b_{km} }} + \|\Sigma\|_{\op}\frac{b_{km} ^{3/2}}{\sqrt{n_m}}\log {n_m K \over b_{km} } \bigg) \bigg)\\
\leq & \sum_{k,m=1}^{K} \sum_{1 \leq q \leq n_m} \Prob\left( \sum_{j=1}^{q} B^{(k,m)}_{(j)} \geq C_{6}\, L^2\bigg(\|\Sigma\|_{\HS}\frac{q }{\sqrt{n_m}}   \sqrt{\log {n_m K \over q }} + \|\Sigma\|_{\op}\frac{q^{3/2}}{\sqrt{n_m}}\log {n_m K \over  q } \bigg)  \right) \\
\leq & \sum_{k,m=1}^{K} \sum_{q=1}^{n_m}  \left( {e n_m \over q} \right)^{q} \exp\left( -{C_{6}^{2}} \,q \log \left( {n_m K \over q} \right) \right) \\
\leq & K^{2} \min_m \,\sum_{q=1}^{n_m}   \exp \left( -C_{7} \,q \log \left( {n_m K \over q} \right) \right) \leq {C_{8} K^{2} \over (\underline{n}K)^{4} } \leq {C_{8} \over {c'_0}^{2} n^{2}},
\end{align*}
where the last inequality is due to $\underline{n}^{2} K \geq c'_0 n$. Thus, we obtain that with probability at least $1-C_{8} /n^{2}$ that
\begin{align*}
\langle Z^{*}T_{3}, \hat{Z}-Z^{*} \rangle \leq C_6\,L^2\sum_{k,m=1}^{K} \bigg(\|\Sigma\|_{\HS}\frac{b_{km} }{\sqrt{n_m}}   \sqrt{\log {n_m K \over b_{km} }} + \|\Sigma\|_{\op}\frac{b_{km} ^{3/2}}{\sqrt{n_m}}\log {n_m K \over b_{km} } \bigg).
\end{align*}
Recall that $\sum_{k,m=1}^K b_{km} = | Z^{*}-Z^{*} \hat{Z}|_1 =S$. Since functions $x^{-1/2}\log x$ and $x^{-1/2}\sqrt{\log x}$ are monotonically decreasing for $x\geq e^2$, we obtain from Jensen's inequality that 
\begin{align*}
\langle Z^{*}T_{3}, \hat{Z}-Z^{*} \rangle \leq C_6\,L^2\,\frac{S}{\sqrt{\underline{n}}}  \,\bigg(\|\Sigma\|_{\HS} \sqrt{\log {\underline{n} K^3 \over S }} + \|\Sigma\|_{\op}\,\sqrt{S}\,\log {\underline{n} K^3 \over S } \bigg).
\end{align*}
By the cyclic invariance of trace and the symmetry of $T_3$ and $\hat Z-Z^\ast$, the same bound holds for $\langle T_{3}Z^{*}, \hat{Z}-Z^{*} \rangle = \langle Z^{*}T_{3}, \hat{Z}-Z^{*} \rangle$.
In addition, the term $\langle Z^{*}T_{3}Z^{*}, \hat{Z}-Z^{*} \rangle=\langle Z^{*}T_{3}, Z^*(\hat{Z}-Z^{*} )Z^*\rangle$ can be handled in the same way
as $\langle Z^{*}T_{3}, \hat{Z}-Z^{*} \rangle$, by noticing that $|Z^*(\hat{Z}-Z^{*} )Z^*|_1 = |Z^*(\hat{Z}-Z^{*} )|_1$ according to Lemma \ref{lem:some_ineq_feasible_set}.

Put all pieces together, we obtain that with probability at least $1-c/n^2$ that
\begin{align*}
\langle T_{3}, \hat{Z}-Z^{*} \rangle \leq
\frac{1}{2}\,\langle T_{1},  Z^{*} -\hat{Z} \rangle+ 3C_6\,L^2 \,S\,\frac{1}{\sqrt{\underline{n}}}  \,\bigg(\|\Sigma\|_{\HS} \sqrt{\log {\underline{n} K^3 \over S }} + \|\Sigma\|_{\op}\,\sqrt{S}\,\log {\underline{n} K^3 \over S } \bigg).
\end{align*}

{\bf Step 5: conclude.} Now we combine the bounds in Step 1 -- 4 to obtain that
\begin{align}
\nonumber
\frac{1}{2}\,\langle T_{1}, Z^{*} -\hat{Z}\rangle \leq&\, 2 C_{1} \,L\,\sqrt{\langle T_{1}, Z^{*}-\hat{Z} \rangle} \sqrt{\|\Sigma\|_{\op} \,S\log\left({2nK^{3} \over S}\right)} \\
\label{eqn:a_quadratic_inequality}
&\,+3C_6\,L^2 \,S\,\frac{1}{\sqrt{\underline{n}}}  \,\bigg(\|\Sigma\|_{\HS} \sqrt{\log {\underline{n} K^3 \over S }} + \|\Sigma\|_{\op}\,\sqrt{S}\,\log {\underline{n} K^3 \over S } \bigg).
\end{align}
holds with probability at least $1-c/n^2$, where recall that $S=| Z^{*}-Z^{*} \hat{Z}|_1$.
According to equation \eqref{eqn:rate_SDP_kernel_kmeans_general_case_primative_step1.1} in Step 1 and equation \eqref{eqn:ineq_3_feasible_set} in Lemma \ref{lem:some_ineq_feasible_set}, we have
$\langle T_{1}, Z^{*} -\hat{Z}\rangle \geq \Delta^2 S/4 \geq 0$. Then solution of the quadratic inequality~\eqref{eqn:a_quadratic_inequality} for $\sqrt{\langle T_{1}, Z^{*} -\hat{Z}\rangle}$ implies
\begin{align}\label{eqn:combining_bound}
\Delta^2 \leq C_9\,L^2\, \|\Sigma\|_{\op} \, \log\left({2nK^{3} \over S}\right) +C_9 \,L^2 \,\frac{1}{\sqrt{\underline{n}}}  \,\bigg(\|\Sigma\|_{\HS} \sqrt{\log {\underline{n} K^3 \over S }} + \|\Sigma\|_{\op}\,\sqrt{S}\,\log {\underline{n} K^3 \over S } \bigg).
\end{align}
This inequality combined with $S \leq |Z^\ast-\hat Z|_1$ due to~\eqref{eqn:ineq_3_feasible_set} and the trivial upper bound $|Z^\ast-\hat Z|_1\leq 2n$ imply
\begin{align*}
\Delta^2 \leq 3C_9\, L^2 \, \|\Sigma\|_{\op} \, \sqrt{\frac{n}{\underline{n}}}\, \log\left({2nK^{3} \over S}\right) + C_9 \,L^2 \,\frac{1}{\sqrt{\underline{n}}}  \,\|\Sigma\|_{\HS} \sqrt{\log {\underline{n} K^3 \over S }}.
\end{align*}
As a consequence, we have
\begin{align*}
S \leq 2nK^3\, \exp \left( -C_{10} \,\bigg(\sqrt{\frac{\underline{n}}{n}}\, \frac{\Delta^2}{L^2\,\|\Sigma\|_{\op}} \wedge \frac{\underline{n}\,\Delta^4}{L^4\,\|\Sigma\|_{\HS}^2}\bigg)\right)\leq 2n K^3 \exp(-C_{11}\,\sqrt{n/\underline{n}}\,) \leq \underline{n},
\end{align*}
where we have used in the second last step our condition that $\SNR^2 \geq  c_0\, n/\underline{n}\geq c_0\, K$ for sufficiently large constant $c_0$.
Now combining the preceding display with inequality \eqref{eqn:combining_bound}, we obtain
\begin{align*}
\Delta^2 \leq 3C_9\, L^2 \, \|\Sigma\|_{\op} \,  \log\left({2nK^{3} \over S}\right) + C_9 \,L^2  \,\frac{1}{\sqrt{\underline{n}}} \,\|\Sigma\|_{\HS} \sqrt{\log {\underline{n} K^3 \over S }}.
\end{align*}
Finally, this inequality combined with equation \eqref{eqn:ineq_3_feasible_set} in Lemma \ref{lem:some_ineq_feasible_set} implies the desired bound
\begin{align*}
|\hat Z - Z^\ast|_1 \leq \frac{2n}{\underline{n}} S \leq C_{12}\, n^2\, K^3/\underline{n}\,\exp(-C_{10} \,\SNR^2) \leq C_{12}\,\exp(-C_{13}\, \SNR^2),
\end{align*}
where the last step is due to the lower bound condition $\SNR^2 \geq  c_0\, n/\underline{n}$.
\end{proof}

\begin{proof}[Proof of Corollary~\ref{coro:exact_recovery}]
For easy presentation, we consider the equal-size clusters case where $n_1=\ldots=n_K=\underline{n}$ and $G_k^\ast=\{(k-1)\underline{n}, (k-1)\underline{n}+1,\ldots,k\underline{n}\}$ for $k=1,\ldots,K$ by reordering the indices. Under this setup, we have 
\begin{equation*}
Z_{ij}^{*} = \left\{
\begin{array}{cc}
1/\underline{n} & \text{if } i, j \in G_{k}^{*} \\
0 & \text{otherwise} \\
\end{array}
\right. .
\end{equation*}
Take $c_1$ large enough so that the upper bound in Theorem~\ref{thm:rate_SDP_kernel_kmeans_general_case} satisfies $C_1 \exp(-C_{2} \SNR^{2}) \leq \frac{1}{3\underline{n}}$.
We use induction to prove that $\hat G_k=G_k^\ast$ at each step for each $k=1,\ldots,K$, which also implies $\hat K=K$. In fact, at $k=1$, since $\max_{i}|\hat Z_{1i} - Z^\ast_{1i}| \leq |\hat Z- Z^\ast|_1 \leq \frac{1}{3\underline{n}}$, we must have $\hat Z_{1i}\in \big[\frac{2}{3\underline{n}},\,\frac{4}{3\underline{n}}\big]$ for $i\in G^\ast_1$ and $\hat Z_{1i}\leq \frac{1}{3\underline{n}}$ for $i\not\in G^\ast_1$ according to the definition of $Z^\ast$. This implies $\hat G_1=G^\ast_1$ according to the choice of $\hat G_1$ in the algorithm. Similarly, assume $\hat G_l=G^\ast_l$ for all $l\leq  k$, then $[n]\setminus \bigcup_{l=1}^k \hat {G}_l=\{k\underline{n}+1,\,k\underline{n}+2,\ldots,n\}$ and $j_{k+1}=k\underline{n}+1$ by definition. Then the fact that $\max_{i}|\hat Z_{j_{k+1}i} - Z^\ast_{j_{k+1}i}| \leq |\hat Z- Z^\ast| \leq \frac{1}{3\underline{n}}$ and the definition of $Z^\ast$ imply $\hat Z_{j_{k+1}i}\in \big[\frac{2}{3\underline{n}},\,\frac{4}{3\underline{n}}\big]$ for $i\in G^\ast_{k+1}$ and $\hat Z_{1i}\leq \frac{1}{3\underline{n}}$ for $i\not\in G^\ast_{k+1}$. Consequently, we must have $\hat G_{k+1} = G^\ast_{k+1}$ according to the choice of $\hat G_{k+1}$ in the algorithm. This completes the proof by induction.
\end{proof}

\appendix

\section{Auxiliary results}
In this section, we collect and prove all auxiliary results in the paper.

\subsection{Feature maps in reproducing kernel Hilbert spaces}\label{sec:feature_map}
In this subsection, we provide a concrete construction of the feature map in kernel clustering.
 To this end, we invoke the theory of reproducing kernel Hilbert space (RKHS). For a detailed survey of linear operators on Hilbert spaces with statistical applications, we refer to the text \cite{HsingEubank2015_Wiley} as an excellent monograph. 

Let the bivariate function $\rho : \bX \times \bX \to \bR$ be a symmetric and positive definite kernel; namely, $\sum_{i,j=1}^{m} c_{i} c_{j} \rho(x_{i},x_{j}) \geq 0$ for all $m \geq 1, x_{1},\dots,x_{m} \in \bX$, and $c_{1},\dots,c_{m} \in \bR$. By the Moore-Aronszajn Theorem (cf. Theorem 2.7.4 in \cite{HsingEubank2015_Wiley}), there exists a unique Hilbert space $\bH := \bH(\rho)$ of real-valued functions on $\bX$ with $\rho$ as its reproducing kernel, i.e., 
\begin{enumeratei}
\item for every $x \in \bX$, $\rho(\cdot,x) \in \bH$;
\item for every $f \in \bH$ and $x \in \bX$, $f(x) = \langle f, \rho(\cdot,x) \rangle$, where $\langle \cdot, \cdot \rangle$ is the inner product of $\bH$. 
\end{enumeratei}
Property (i) defines a feature map $\phi : \bX \to \bH$ via $ x \mapsto \rho(\cdot,x)$, which is known in literature as the RKHS map \cite{Aronszajn1950_TAMS}. Property (ii) shows that $\rho$ satisfies the reproducing kernel property for all functions in the Hilbert space $\bH$. Thus $\bH$ is the RKHS associated with $\rho$. It is immediate from these two properties that 
\[
\rho(x,y)=\langle \rho(\cdot,y), \rho(\cdot,x) \rangle = \langle \phi(x), \phi(y) \rangle \quad \forall x,y \in \bX.
\]
Then the similarity matrix $A $ is chosen $a_{ij} = \rho(X_{i},X_{j}) = \langle \phi(X_{i}), \phi(X_{j}) \rangle$. Statistical properties of the SDP solution $\hat{Z}$ for (\ref{eqn:clustering_Kmeans_sdp}) rely on the distribution of the feature vectors $\phi(X_{i})$ in $\bH$, which is a special case of Theorem~\ref{thm:rate_SDP_kernel_kmeans_general_case}.


\subsection{Auxiliary proofs and lemmas}\label{sec:auxiliary_lemmas}
In this subsection, we provide additional proofs of the technical results used in the paper. 

\begin{proof}[Proof of Lemma \ref{lem:hilbert_space_gaussian_rv}]
Without loss of generality, we may assume $\mu=0$. Suppose that $Z \sim N(0,\Gamma)$. Then $\|Z\|_{\psi_{2}, \Gamma} = 1$ is obvious from Definition \ref{defin:subgaussian_Hilbert_space} and \ref{defin:gaussian_Hilbert_space}. Let $M(t) = \E[e^{t \langle z, Z \rangle}], t \in \bR,$ be the moment generating function of $\langle z, Z \rangle$. Then Taylor's expansion yields that 
\[
\left. {d^{2} M(t) \over d t^{2}} \right|_{t=0} = \E \langle z, Z \rangle^{2} = \E \langle z, \langle z, Z \rangle Z \rangle = \E \langle z, (Z \otimes Z) z \rangle = \langle z, \E(Z \otimes Z) z \rangle = \langle z, \Sigma z \rangle.
\]
On the other hand, since $Z \sim N(0, \Gamma)$, we have 
\[
{d^{2} M(t) \over d t^{2}} = (1+t^{2}) \langle \Gamma z, z \rangle e^{t^{2} \langle \Gamma z, z \rangle / 2}.
\]
Thus it follows that 
\[
\langle (\Sigma-\Gamma) z, z \rangle = 0 \quad \text{for all } z \in \bH,
\]
which implies that $\Sigma = \Gamma$. Suppose that $Z \sim \subg(\Gamma)$. By Markov's inequality and Definition \ref{defin:subgaussian_Hilbert_space}, we have  
\[
\Prob(\langle z, Z \rangle \geq t) \leq \inf_{\lambda > 0} e^{-\lambda t} \E[e^{\lambda \langle z, Z \rangle}] \leq \inf_{\lambda > 0} e^{-\lambda t + {\alpha^{2} \lambda^{2} \over 2} \langle \Gamma z, z \rangle} = e^{-{t^{2} \over 2 \alpha^{2} \langle \Gamma z, z \rangle}}, 
\]
where $\alpha^{2} = \|Z\|_{\psi_{2}. \Gamma}^{2}$. Then, 
\[
\langle \Sigma z, z \rangle = \E \langle z, Z \rangle^{2} = \int_{0}^{\infty} \Prob( |\langle z, Z \rangle| \geq \sqrt{t}) dt \leq 2 \int_{0}^{\infty} e^{-{t \over 2 \alpha^{2} \langle \Gamma z, z \rangle}} dt = 4 \alpha^{2} \langle \Gamma z, z \rangle.
\]
Thus it is immediate that $\langle (4 \alpha^{2} \Gamma - \Sigma) z, z \rangle \geq 0$ for all $z \in \bH$, i.e., $\Sigma \preceq 4 \|Z\|_{\psi_{2}, \Gamma}^{2} \Gamma$. 
\end{proof}

\begin{lem}[Moment generating function bound for squared norm of a sub-gaussian random variable in $\R^{n}$]
\label{lem:mgf_bound_subgaussian}
Let $\Gamma$ be an $n \times n$ positive semidefinite matrix and $X$ be a random variable in $\R^{n}$ such that $\E[X] = 0$ and $\E[e^{z^{T} X}] \leq e^{z^{T} \Gamma z / 2}$ for all $z \in \R^{n}$. Let $Z \sim N(0, \Gamma)$. Then, 
\[
\E \left[ e^{t \|X\|_{2}^{2} \over 2} \right] \leq \E \left[ e^{t \|Z\|_{2}^{2} \over 2} \right] \quad \forall \; 0 \leq t < \|\Gamma\|_{\op}^{-1}. 
\]
\end{lem}

\begin{proof}[Proof of Lemma \ref{lem:mgf_bound_subgaussian}]
The case for $t = 0$ is obvious. Without loss of generality, we may assume $\Gamma$ is (strictly) positive definite since otherwise we can consider $\Gamma + \delta I_{n}$ for $\delta > 0$ and then let $\delta \to 0$. Consider $t \in (0, \|\Gamma\|_{\op}^{-1})$. Denote the determinant of $\Gamma$ as $|\Gamma|$. Observe that 
\begin{align*}
A :=& {1 \over (2\pi)^{n/2} |\Gamma|^{1/2}} \int_{\bR^{n}} e^{-{\|z\|_{2}^{2} \over 2t}} \E [e^{z^{T} X}] dz \\
=_{(1)}& \E \left[{1 \over (2\pi)^{n/2} |\Gamma|^{1/2}} \int_{\bR^{n}} e^{-{\|z-tX\|_{2}^{2} \over 2t}} dz \; e^{t\|X\|_{2}^{2} \over 2} \right] \\
=_{(2)}& \E \left[ e^{t\|X\|_{2}^{2} \over 2} \right] {1 \over (2\pi)^{n/2} |\Gamma|^{1/2}} \int_{\bR^{n}} e^{-{\|z\|_{2}^{2} \over 2t}} dz \\
=_{(3)}& \E \left[ e^{t\|X\|_{2}^{2} \over 2} \right] {1 \over |t^{-1} \Gamma|^{1/2}}, 
\end{align*}
where $(1)$ follows from Fubini's theorem, $(2)$ from the translational invariance of the Gaussian density integral, and $(3)$ from that the integration of the standard Gaussian distribution $N(0,I_{n})$ equals to one. Thus we get 
\[
\E \left[ e^{t\|X\|_{2}^{2} \over 2} \right] = |t^{-1} \Gamma|^{1/2} A. 
\]
Since $\E[e^{z^{T} X}] \leq e^{z^{T} \Gamma z / 2}$ for all $z \in \R^{n}$, we have for $t \in (0, \|\Gamma\|_{\op}^{-1})$, 
\begin{align*}
A \leq& {1 \over (2\pi)^{n/2} |\Gamma|^{1/2}} \int_{\bR^{n}} e^{-{z^{T} z \over 2t}} e^{z^{T} \Gamma z \over 2} dz \\
=& {1 \over (2\pi)^{n/2} |\Gamma|^{1/2}} \int_{\bR^{n}} e^{-{1 \over 2} z^{T} (t^{-1}I_{n} - \Gamma) z} dz \\
=& {1 \over |\Gamma|^{1/2} |t^{-1}I_{n} - \Gamma|^{1/2}} \left[ {1 \over  (2\pi)^{n/2} |(t^{-1}I_{n} - \Gamma)^{-1}|^{1/2}} \int_{\bR^{n}} e^{-{1  \over 2}z^{T} (t^{-1}I_{n} - \Gamma) z} dz \right] \\
=& {1 \over |\Gamma|^{1/2} |t^{-1}I_{n} - \Gamma|^{1/2}}.
\end{align*}
Then we have 
\[
\E \left[ e^{t\|X\|_{2}^{2} \over 2} \right] \leq {|t^{-1} \Gamma|^{1/2} \over |\Gamma|^{1/2} |t^{-1}I_{n} - \Gamma|^{1/2}} = {1 \over |I_{n} - t \Gamma|^{1/2}}  \quad \forall \; 0 \leq t < \|\Gamma\|_{\op}^{-1}. 
\]
On the other hand, for $Z \sim N(0,\Gamma)$, similar calculations show that  
\begin{align*}
\E \left[ e^{s\|Z\|_{2}^{2} \over 2} \right] =& {1 \over (2\pi)^{n/2} |\Gamma|^{1/2}} \int_{\bR^{n}} e^{-{1  \over 2} z^{T} \Gamma^{-1} z} e^{{s \over 2} z^{T} z} dz \\
=& {1 \over (2\pi)^{n/2} |\Gamma|^{1/2}} \int_{\bR^{n}} e^{-{1 \over 2} z^{T} (\Gamma^{-1} -s I_{n}) z} dz \\
=& {|\Gamma^{-1}(I_{n} - s\Gamma)|^{-1/2} \over |\Gamma|^{1/2}} = {1 \over |I_{n} - s\Gamma|^{1/2}} \quad \forall \; s < \|\Gamma\|_{\op}^{-1}, 
\end{align*}
from which Lemma \ref{lem:mgf_bound_subgaussian} is immediate.  
\end{proof}

\begin{lem}[Upper bound for squared norm of a sub-gaussian random variable in $\R^{n}$]
\label{lem:mgf_bound_subgaussian_centered}
In the setting of Lemma \ref{lem:mgf_bound_subgaussian}, we have 
\begin{equation}
\label{eqn:mgf_bound_subgaussian_centered}
\E \left[ e^{{t \over 2} (\|X\|_{2}^{2}-\tr(\Gamma))} \right] \leq e^{{t^{2} \over 2} \|\Gamma\|_{\HS}^{2}} \quad \forall \; 0 \leq t  < (2 \|\Gamma\|_{\op})^{-1}. 
\end{equation}
Consequently, we have for any $u > 0$, 
\begin{equation}
\label{eqn:Bernstein_bound_subgaussian_centered}
\Prob \left( \|X\|_{2}^{2} - \tr(\Gamma) \geq u \right) \leq \exp \left[ -{1 \over 8} \min \left( {u^{2} \over \|\Gamma\|_{\HS}^{2}}, {u \over \|\Gamma\|_{\op}} \right) \right].
\end{equation}
\end{lem}

\begin{proof}[Proof of Lemma \ref{lem:mgf_bound_subgaussian_centered}]
Let $Z \sim N(0,\Gamma)$. By the calculations in Lemma \ref{lem:mgf_bound_subgaussian}, we have for all $t < \|\Gamma\|_{\op}^{-1}$, 
\[
\E \left[ e^{{t \over 2} (\|Z\|_{2}^{2}-\tr(\Gamma))} \right] = {e^{-{t \over 2} \tr(\Gamma)} \over |I_{n}-t \Gamma|^{1/2}} = \prod_{i=1}^{n} {e^{-t \gamma_{i} / 2} \over \sqrt{1-t \gamma_{i}}},
\]
where $(\gamma_{i})_{i=1}^{n}$ are eigenvalues of $\Gamma$. Using the inequality 
\[
{e^{-t} \over \sqrt{1-2t}} \leq e^{2 t^{2}} \quad \forall |t| < 1/4,
\]
we have 
\[
\E \left[ e^{{t \over 2} (\|Z\|_{2}^{2}-\tr(\Gamma))} \right] \leq \prod_{i=1}^{n} e^{t^{2} \gamma_{i}^{2} \over 2} = e^{t^{2} \|\Gamma\|_{\HS}^{2} \over 2} \quad \forall |t| < (2\|\Gamma\|_{\op})^{-1}.
\]
Combining the last inequality with Lemma \ref{lem:mgf_bound_subgaussian}, we get (\ref{eqn:mgf_bound_subgaussian_centered}). By Markov's inequality, we have for any $u > 0$ and $0 \leq t  < (2 \|\Gamma\|_{\op})^{-1}$, 
\[
\Prob \left( \|X\|_{2}^{2} - \tr(\Gamma) \geq u \right) \leq e^{-{tu \over 2} + {t^{2} \over 2} \|\Gamma\|_{\HS}^{2}}.
\]
Choosing $t = t^{*} := {u \over 2 \|\Gamma\|_{\HS}^{2}} \wedge {1 \over 2 \|\Gamma\|_{\op}}$, we get 
\[
\Prob \left( \|X\|_{2}^{2} - \tr(\Gamma) \geq u \right) \leq \exp \left(-{ut^{*} \over 4} \right) = \exp \left[ -{1 \over 8} \min \left( {u^{2} \over \|\Gamma\|_{\HS}^{2}}, {u \over \|\Gamma\|_{\op}} \right) \right].
\]
\end{proof}

\begin{lem}[Moment generating function bound for centered squared norm of a sub-gaussian random variable in $\bH$]
\label{lem:mgf_bound_subgaussian_hilbert_space}
Let $\Gamma \in \cB(\bH)$ be a positive definite trace class operator on $\bH$. Let $X$ be a centered sub-gaussian random variable in $\bH$ with respect to $\Gamma$ and $L = \|X\|_{\psi_{2}}$. Then, 
\[
\E \left[ e^{{t \over 2} (\|X\|^{2} - L^{2} \|\Gamma\|_{\tr})} \right] \leq e^{{t^{2} L^4 \over 2} \|\Gamma\|_{\HS}^{2}} \quad \forall \;  0 \leq t < {1 \over 2 L^{2} \|\Gamma\|_{\op}}. 
\]
\end{lem}

\begin{proof}[Proof of Lemma \ref{lem:mgf_bound_subgaussian_hilbert_space}]
The proof is a standard approximation argument combined with Lemma \ref{lem:mgf_bound_subgaussian_centered}. Let $(e_{k})_{k=1}^{\infty}$ be a CONS of $\bH$. By Parseval's identity, $\|X\|^{2} = \sum_{k=1}^{\infty} \langle X, e_{k} \rangle^{2}$, where convergence of the sum is made in the $\ell^{2}$ sense. Let $K > 0$ be a finite integer. Put $X_{K} = (\langle X, e_{1} \rangle,\dots,\langle X, e_{K} \rangle)^{T}$. Then $X_{K} \sim \subg(L^{2} \Gamma_{K})$ is a mean-zero random variable in $\bR^{n}$ with $\Gamma_{K,jk} = \langle \Gamma e_{j}, e_{k}\rangle$ for $j,k=1,\dots,K$. Since $\|\Gamma_{K}\|_{\op} \leq \|\Gamma\|_{\op}$, it follows from Lemma \ref{lem:mgf_bound_subgaussian_centered} that 
\[
\E \left[ e^{{t \over 2} (\|X_{K}\|^{2} - L^{2} \|\Gamma_{K}\|_{\tr})} \right] \leq e^{{t^{2}L^4 \over 2} \|\Gamma_{K}\|_{\HS}^{2}} \quad \forall \; 0 \leq t < {1 \over L^{2} \|\Gamma\|_{\op}}. 
\]
Letting $K \to \infty$, we have $\|X_{K}\|_{2}^{2} \nearrow \|X\|^{2}$, $\tr(\Gamma_{K}) = \|\Gamma_{K}\|_{\tr} \nearrow \|\Gamma\|_{\tr}$, and $\|\Gamma_{K}\|_{\HS}^{2} \nearrow \|\Gamma\|_{\HS}^{2}$. Then Lemma \ref{lem:mgf_bound_subgaussian_hilbert_space} follows from the monotone convergence theorem. 
\end{proof}

\begin{lem}[Squared norm of a sub-gaussian random variable in $\bH$ is sub-exponential]
\label{lem:squared_norm_is_subexp}
Let $\Gamma \in \cB(\bH)$ be a positive definite trace class operator on $\bH$ and $X$ be a centered $\subg(\Gamma)$ random variable in $\bH$. Then there exists a universal constant $C > 0$ such that 
\[
\left\| \|X\|^{2} \right\|_{\psi_{1}} \leq C \|X\|_{\psi_{2}}^{2} \|\Gamma\|_{\tr}. 
\]
Thus $\|X\|^{2}$ is a sub-exponential random variable in $\bR$. 
\end{lem}

\begin{proof}[Proof of Lemma \ref{lem:squared_norm_is_subexp}]
Let $(e_{k})_{k=1}^{\infty}$ be a CONS of $\bH$. By Parseval's identity, $\|X\|^{2} = \sum_{k=1}^{\infty} \langle X, e_{k} \rangle^{2}$. Since $\|\cdot\|_{\psi_{1}}$ for real-valued random variables is a norm, we have by triangle inequality that 
\[
\left\| \|X\|^{2} \right\|_{\psi_{1}} \leq \sum_{k=1}^{\infty} \left\| \langle X, e_{k} \rangle^{2} \right\|_{\psi_{1}} = \sum_{k=1}^{\infty} \left\| \langle X, e_{k} \rangle \right\|_{\psi_{2}}^{2},
\]
where the last step follows from Lemma 2.7.6 in \cite{Vershynin2018_Cambridge}. Since $X \sim \subg(\Gamma)$ with mean zero, we have for any $\lambda > 0$, 
\[
\E \left[ e^{\lambda \langle X, e_{k} \rangle} \right] \leq e^{{\lambda^{2} \over 2} \|X\|_{\psi_{2}}^{2} \langle \Gamma e_{k}, e_{k} \rangle},
\]
which implies that there exists a universal constant $C > 0$ such that 
\[
\| \langle X, e_{k} \rangle \|_{\psi_{2}} \leq C \|X\|_{\psi_{2}} \sqrt{\langle \Gamma e_{k}, e_{k} \rangle}.
\]
Then, 
\[
\left\| \|X\|^{2} \right\|_{\psi_{1}} \leq \sum_{k=1}^{\infty} C^{2} \|X\|_{\psi_{2}}^{2} \langle \Gamma e_{k}, e_{k} \rangle = C^{2} \|X\|_{\psi_{2}}^{2} \|\Gamma\|_{\tr}.
\]
\end{proof}

Let $r$ be a non-negative integer and $0 \leq f < 1$. For $s = r+f \geq 0$, we define the (generalized) sum
\[
\sum_{i=1}^{s} a_{i} := \sum_{i=1}^{r} a_{i} + f a_{r+1}.
\]

\begin{lem}[Monotone rearrangement]
\label{lem:ordered_sum_inequality}
For any $a_{1},\dots,a_{n} \in \bR$ and $b_{1},\dots,b_{n} \in [0,1]$, we have 
\[
\sum_{i=1}^{n} a_{i} b_{i} \leq \sum_{i=1}^{s} a_{(i)},
\]
where $a_{(1)} \geq \cdots \geq a_{(n)}$ and $s = \sum_{i=1}^{n} b_{i}$.
\end{lem}

By definition, we clearly have $\sum_{i=1}^{s} a_{i} \leq \max\{\sum_{i=1}^{r} a_{i}, \, \sum_{i=1}^{r+1} a_{i}\}$, and for $0 \leq s \leq 1$, $\sum_{i=1}^{s} = s a_1 \leq s a_{(1)}$. Moreover, Lemma~\ref{lem:ordered_sum_inequality} is tighter than the classical inequality $\sum_{i=1}^{n} a_i b_i \leq |a|_{\infty} |b|_1$ because $a_{(i)} \leq |a|_{\infty}$. Using the order statistics structure, we are able to obtain the exponential decay of error result in the $K$-means SDP clustering problem in Section~\ref{subsec:rate_SDP_kernel_Kmeans}.

\begin{proof}[Proof of Lemma~\ref{lem:ordered_sum_inequality}]
Write $s = r + f$, where $r$ is a non-negative integer and $f \in [0,1)$. Let $X$ be a random variable taking values in $\{a_{1},\dots,a_{n}\}$ with the probability mass function $\Prob(X = a_{i}) = b_{i} / s$. Let $Y$ be another random variable taking values in $\{a_{(1)},\dots,a_{(n)}\} = \{a_{1},\dots,a_{n}\}$ with the probability mass function $\Prob(Y = a_{(j)}) = 1/s$ for $1 \leq j \leq r$, $\Prob(Y = a_{(r+1)}) = f/s$, and $\Prob(Y = a_{(j)}) = 0$ for $r+2 \leq j \leq n$. Since $b_{i} \in [0, 1]$, we can always shift a non-negative proportion of mass from $X$ to $Y$. Thus we have $\E[X] \leq \E[Y]$ and the lemma follows.
\end{proof}

\begin{lem}
\label{lem:some_ineq_feasible_set}
Let $Z^{*}$ be defined in (\ref{eqn:Kmeans_true_membership_matrix}). Then for any $Z \in \sC$ defined in (\ref{eqn:clustering_Kmeans_sdp}), we have 
\begin{align}
\label{eqn:ineq_1_feasible_set}
|Z^{*} - Z^{*} Z Z^{*}|_{1} = |Z^{*} - Z^{*} Z|_{1} =& 2 \sum_{1 \leq k \neq m \leq K} |Z_{G_{k}^{*} G_{m}^{*}}|_{1}, \\
\label{eqn:ineq_2_feasible_set}
\|(I-Z^{*}) Z (I-Z^{*})\|_{\tr} \leq& {|Z^{*} - Z^{*} Z|_{1} \over 2 \underline{n}}, \\
\label{eqn:ineq_3_feasible_set}
|Z^{*} - Z^{*} Z|_{1} \leq |Z^{*} - Z|_{1} \leq& {2 n \over \underline{n}} |Z^{*} - Z^{*} Z|_{1}.
\end{align}
\end{lem}

\begin{proof}[Proof of Lemma \ref{lem:some_ineq_feasible_set}]
See Lemma 1 in \cite{GiraudVerzelen2018}. 
\end{proof}

\section*{Acknowledgements}
The authors would like to thank an anonymous referee and an Associate Editor for their many careful comments that improved the quality of this paper. X. Chen's research was supported in part by NSF DMS-1404891, NSF CAREER Award DMS-1752614, UIUC Research Board Awards (RB17092, RB18099), and a Simons Fellowship. Y. Yang's research was supported in part by NSF DMS-1810831.

\bibliographystyle{plain}
\bibliography{clustering_sdp}

\begin{thebibliography}{10}

\bibitem{Adamczak2015_ECP}
Rados{\l}aw Adamczak.
\newblock A note on the hanson-wright inequality for random vectors with
  dependencies.
\newblock {\em Electronic Communications in Probability}, 20(72):1--13, 2015.

\bibitem{Antonini1997_RSMUP}
Rita~Giuliano Antonini.
\newblock {Subgaussian random variables in Hilbert spaces}.
\newblock {\em Rend. Sem. Mat. Univ. Padova}, 98:89--99, 1997.

\bibitem{Aronszajn1950_TAMS}
N~Aronszajn.
\newblock Theory of reproducing kernels.
\newblock {\em Trans. Amer. Math. Soc.}, 68:337--404, 1950.

\bibitem{BartheMilman2013_CMP}
Franck Barthe and Emanuel Milman.
\newblock Transference principles for log-sobolev and spectral-gap with
  applications to conservative spin systems.
\newblock {\em Communications in Mathematical Physics}, 323(2):575--625, 2013.

\bibitem{BuneaGiraudRoyerVerzelen2016}
Florentina Bunea, Christophe Giraud, Martin Royer, and Nicolas Verzelen.
\newblock {PECOK: a convex optimization approach to variable clustering}.
\newblock {\em arXiv:1606.05100}, 2016.

\bibitem{chen2018a}
Xiaohui Chen.
\newblock Gaussian and bootstrap approximations for high-dimensional
  {U}-statistics and their applications.
\newblock {\em Annals of Statistics}, 46(2):642--678, 2018.

\bibitem{FeiChen2018}
Yingjie Fei and Yudong Chen.
\newblock Hidden integrality of sdp relaxation for sub-gaussian mixture models.
\newblock {\em arXiv:1803.06510}, 2018.

\bibitem{ferraty2006nonparametric}
Fr{\'e}d{\'e}ric Ferraty and Philippe Vieu.
\newblock {\em Nonparametric functional data analysis: theory and practice}.
\newblock Springer Science \& Business Media, 2006.

\bibitem{FilipponeCamastraMasulliRovetta2008_PR}
Maurizio Filippone, Francesco Camastra, Franscesco Masulli, and Stefano
  Rovetta.
\newblock A survey of kernel and spectral methods for clustering.
\newblock {\em Pattern Recognition}, 41:176--190, 2008.

\bibitem{GiraudVerzelen2018}
Christophe Giraud and Nicolas Verzelen.
\newblock Partial recovery bounds for clustering with the relaxed $k$means.
\newblock {\em arXiv:1807.07547}, 2018.

\bibitem{HansonWright1971_AMS}
D.L. Hanson and E.T. Wright.
\newblock {A bound on tail probabilities for quadratic forms in independent
  random variables}.
\newblock {\em The Annals of Mathematical Statistics}, 42:1079--1083, 1971.

\bibitem{HsingEubank2015_Wiley}
Tailen Hsing and Randall Eubank.
\newblock {\em {Theoretical Foundations of Functional Data Analysis, with an
  Introduction to Linear Operators}}.
\newblock Wiley Series in Probability and Statistics. Wiley, 2015.

\bibitem{HsuKakadeZhang2012_ECP}
Daniel Hsu, Sham Kakade, and Tong Zhang.
\newblock A tail inequality for quadratic forms of subgaussian random vectors.
\newblock {\em Electronic Communications in Probability}, 17(52):1--6, 2012.

\bibitem{ieva2013multivariate}
Francesca Ieva, Anna~M Paganoni, Davide Pigoli, and Valeria Vitelli.
\newblock Multivariate functional clustering for the morphological analysis of
  electrocardiograph curves.
\newblock {\em Journal of the Royal Statistical Society: Series C (Applied
  Statistics)}, 62:401--418, 2013.

\bibitem{LaurentMassart2000_AoS}
B.~Laurent and Massart P.
\newblock Adaptive estimation of a quadratic functional by model selection.
\newblock {\em Ann. Statist.}, 28(5):1302--1338, 2000.

\bibitem{LiLiLingStohmerWei2017}
Xiaodong Li, Yang Li, Shuyang Ling, Thomas Stohmer, and Ke~Wei.
\newblock When do birds of a feather flock together? $k$-means, proximity, and
  conic programming.
\newblock {\em arXiv:1710.06008}, 2017.

\bibitem{Lloyd1982_TIT}
Stuart Lloyd.
\newblock Least squares quantization in pcm.
\newblock {\em IEEE Transactions on Information Theory}, 28:129--137, 1982.

\bibitem{PengWei2007_SIAMJOPTIM}
Jiming Peng and Yu~Wei.
\newblock Approximating $k$-means-type clustering via semidefinite programming.
\newblock {\em SIAM J. OPTIM}, 18(1):186--205, 2007.

\bibitem{rasmussen2004gaussian}
Carl~Edward Rasmussen.
\newblock Gaussian processes in machine learning.
\newblock In {\em Advanced lectures on machine learning}, pages 63--71.
  Springer, 2004.

\bibitem{RauhutFoucart2013}
Holger Rauhut and Simon Foucart.
\newblock {\em {A Mathematical Introduction to Compressive Sensing}}.
\newblock Applied and Numerical Harmonic Analysis. BirkH\"auser, 2013.

\bibitem{Royer2017_NIPS}
Martin Royer.
\newblock Adaptive clustering through semidefinite programming.
\newblock In I.~Guyon, U.~V. Luxburg, S.~Bengio, H.~Wallach, R.~Fergus,
  S.~Vishwanathan, and R.~Garnett, editors, {\em Advances in Neural Information
  Processing Systems 30}, pages 1795--1803. Curran Associates, Inc., 2017.

\bibitem{RudelsonVershynin2013_ECP}
Mark Rudelson and Roman Vershynin.
\newblock {Hanson-Wright inequality and sub-Gaussian concentration}.
\newblock {\em Electronic Communications in Probability}, 18(82):1--9, 2013.

\bibitem{ScholkopfSmola2001_LearningKernels}
Bernhard Sch{\"o}lkopf and Alexander Smola.
\newblock {\em Learning with Kernels: Support Vector Machines, Regularization,
  Optimization, and Beyond}.
\newblock The MIT Press, 1 edition, 2001.

\bibitem{SongSmolaGrettonBorgwardt2007_ICML}
Le~Song, Alex Smola, Arthur Gretton, and Karsten Borgwardt.
\newblock A dependence maximization view of clustering.
\newblock In {\em Proceedings of the 24th International Conference on Machine
  Learning}, 2007.

\bibitem{tarpey2003clustering}
Thaddeus Tarpey and Kimberly~KJ Kinateder.
\newblock Clustering functional data.
\newblock {\em Journal of classification}, 20:093--114, 2003.

\bibitem{van2008reproducing}
Aad~W van~der Vaart, J~Harry van Zanten, et~al.
\newblock Reproducing kernel hilbert spaces of gaussian priors.
\newblock In {\em Pushing the limits of contemporary statistics: contributions
  in honor of Jayanta K. Ghosh}, pages 200--222. Institute of Mathematical
  Statistics, 2008.

\bibitem{Vershynin2012_nonasymptotic}
Roman Vershynin.
\newblock {\em Introduction to the non-asymptotic analysis of random matrices},
  pages 210--268.
\newblock Compressed sensing. Cambridge University Press, 2012.

\bibitem{Vershynin2018_Cambridge}
Roman Vershynin.
\newblock {\em {High-Dimensional Probability: An Introduction with Applications
  in Data Science}}.
\newblock Cambridge Series in Statistical and Probabilistic Mathematics.
  Cambridge University Press, 2018.

\bibitem{VuWang2015_RSA}
Van Vu and Ke~Wang.
\newblock Random weighted projections, random quadratic forms and random
  eigenvectors.
\newblock {\em Random Structures \& Algorithms}, 47(4):792--821, 2015.

\bibitem{Wainwright2019_HDS}
Martin~J. Wainwright.
\newblock {\em {High-Dimensional Statistics: A Non-Asymptotic Viewpoint}}.
\newblock Cambridge Series in Statistical and Probabilistic Mathematics.
  Cambridge University Press, 2019.

\bibitem{Wright1973_AoP}
E.T. Wright.
\newblock {A bound on tail probabilities for quadratic forms in independent
  random variables whose distributions are not necessarily symmetric}.
\newblock {\em The Annals of Probability}, 1:1068--1070, 1973.

\end{thebibliography}

\end{document}